\documentclass{siamart190516}

\usepackage{amsmath, amssymb}
\usepackage{lipsum}
\usepackage{amsfonts}

\usepackage{epstopdf}
\usepackage{algorithmic}
\usepackage{stmaryrd}
\usepackage{bbm}
\usepackage{cases}
\usepackage{hyperref}
\usepackage[T1]{fontenc}
\usepackage[latin1]{inputenc}
\usepackage{upgreek}
\usepackage[english]{babel}
\usepackage[cyr]{aeguill}
\usepackage{xspace}
\usepackage{graphicx}
\usepackage{caption}
\usepackage{color}
\usepackage[labelformat=empty]{caption}
\usepackage{epigraph}
\usepackage{etoolbox}
\usepackage{enumitem}
\usepackage{stmaryrd}
\usepackage[margin=1in]{geometry}
\usepackage{cite}
\usepackage{dsfont}
\usepackage{ushort}
\usepackage{xcolor}

\usepackage{tikz}
\usetikzlibrary{arrows,backgrounds}
\usepgflibrary{shapes.multipart}
\usetikzlibrary{arrows.meta}

\newcommand{\ol}{\overline}
\newcommand{\ul}{\underline}
\newcommand{\eps}{\varepsilon}

\newcommand{\ba}{\begin{array}}
\newcommand{\ea}{\end{array}}
\newcommand{\be}{\begin{equation}}
\newcommand{\ee}{\end{equation}}
\newcommand{\bea}{\begin{eqnarray}}
\newcommand{\eea}{\end{eqnarray}}
\newcommand{\beaa}{\begin{eqnarray*}}
\newcommand{\eeaa}{\end{eqnarray*}}

\def\dbE{\mathbb{E}}

\def\dbN{\mathbb{N}}
\def\dbP{\mathbb{P}}
\def\dbR{\mathbb{R}}

\def\dbQ{\mathbb{Q}}

\def\a{\alpha}
\def\b{\beta}
\def\g{\gamma}
\def\d{\delta}
\def\e{\varepsilon}
\def\z{\zeta}
\def\k{\kappa}
\def\l{\lambda}

\def\si{\sigma}

\def\f{\varphi}

\def\D{\Delta}
\def\Th{\Theta}

\def\Si{\Sigma}

\def\O{\Omega}
\def\cA{{\cal A}}

\def\cF{{\cal F}}
\def\cG{{\cal G}}

\def\cI{{\cal I}}

\def\cK{{\cal K}}
\def\cL{{\cal L}}
\def\cM{{\cal M}}

\def\cP{{\cal P}}

\def\cS{{\cal S}}

\def\cV{{\cal V}}
\def\cW{{\cal W}}

\def\cZ{{\cal Z}}

\def\E{{\mathsf{E}}}

\def\Ent{{\mathsf{Ent}}}

\def\syc{{\rm sc}}
\def\rc{{\rm rc}}

\def\no{\noindent}

\def\ms{\medskip}

\def\q{\quad}

\def\pa{\partial}
\def\cd{\cdot}
\def\cds{\cdots}

\def\qed{ \hfill \vrule width.25cm height.25cm depth0cm\smallskip}

\newcommand{\basa}{\begin{assumption}}
\newcommand{\easa}{\end{assumption}}

\newcommand{\bas}{\begin{assum}}
\newcommand{\eas}{\end{assum}}

\def\limsup{\mathop{\overline{\rm lim}}}

\def\pa{\partial}

\def\wt{\widetilde}
 \def\cd{\cdot}
\def\cds{\cdots}

\def\wt{\widetilde}

\def\1{{\bf 1}}

\def\:{\!:\!}

at 9pt

\newtheorem{thm}{Theorem}
\newtheorem{lem}[thm]{Lemma}
\newtheorem{cor}[thm]{Corollary}
\newtheorem{prop}[thm]{Proposition}
\newtheorem{rem}[thm]{Remark}

\newtheorem{assum}[thm]{Assumption}

\DeclareMathOperator*{\argmin}{\arg\!\min}

\begin{document}

	\renewcommand {\theequation}{\arabic{section}.\arabic{equation}}
	\def\thesection{\arabic{section}}
	
	\numberwithin{equation}{section}
	\numberwithin{thm}{section}

	\title{\textbf{Mean-field Langevin System, Optimal Control and Deep Neural Networks}}
	\author{Kaitong HU \footnote{CMAP, Ecole Polytechnique, IP Paris, 91128 Palaiseau Cedex, France, kaitong.hu@polytechnique.fr.}\and
	Anna KAZEYKINA \footnote{Laboratoire de Math\'ematiques d'Orsay, Universit\'e Paris-Sud, CNRS, Universit\'e
Paris-Saclay, 91405 Orsay, France,  anna.kazeykina@math.u-psud.fr.}\and
	Zhenjie REN\footnote{Universit\'e Paris-Dauphine, PSL Research University, CNRS, UMR [7534], Ceremade, 75016 Paris, France, ren@ceremade.dauphine.fr.} }
	\date{}
	\maketitle{}

\begin{abstract}
In this paper, we study a regularised relaxed optimal control problem and, in particular, we are concerned with the case where the control variable is of large dimension.  We introduce a system of mean-field Langevin equations, the invariant measure of which is shown to be the optimal control of the initial problem under mild conditions. Therefore, this system of processes can be viewed as a continuous-time numerical algorithm for computing the optimal control. As an application, this result endorses the solvability of the stochastic gradient descent algorithm for a wide class of deep neural networks. 
\end{abstract}
	
\paragraph*{Keywords:}Mean-Field Langevin Dynamics, Gradient Flow, Neural Networks 

\paragraph*{MSC:} 60H30, 37M25 	
	
\section{Introduction}

This paper revisits the classical optimal control problem, that is,
\bea\label{standardcontrol}
\inf_\a V^0(\a), \q\mbox{where}\q V^0(\a):= \int_0^T L(t, X_t^\a , \a_t) dt +G(X^\a_T)
\q\mbox{and}\q X_t = x_0+ \int_0^t \phi(r, X^\a_r, \a_r)dr.
\eea
In particular, we aim at providing a feasible algorithm for solving such problem (indeed, its regularized version) when the dimensions of  the state $X$ and of the control $\a$ are both large.

It has been more than half a century since the discovery of Pontryagin's maximum principle \cite{BGP60}, which states that in order to be an optimal control to the problem \eqref{standardcontrol}, $\a^*$ needs to satisfy the forward-backward ODE system:
\bea
\label{eq:MP}
\left\{
\begin{array}{ll}
\a^*_t = \argmin_a H(t, X^*_t, a, P^*_t), \q\mbox{where}\q H(t,x,a,p):= L(t,x,a) + p\cd \phi(t,x,a),\\
X^*_t =x_0 + \int_0^t  \phi(r, X^*_r, \a^*_r)dr, \\
P^*_t =  \nabla_x G(X^*_T) + \int_0^t \nabla_x H(r, X^*_r, \a^*_r, P^*_r) dr.\\
\end{array}
\right.
\eea
It is worth mentioning that this necessary condition becomes sufficient if one imposes convexity condition on the coefficients. To solve the forward-backward system, the most naive way is to follow a fixed-point algorithm, that is, starting with an arbitrary control $\a$, evaluate the forward equation and then the backward one, and eventually compute a new control $\tilde \a$ by solving the optimization problem on the top line. Under some mild conditions, one may show that this mapping $\a\mapsto \tilde \a$ is a contraction at least on short horizon (i.e. for small $T$), see e.g. \cite{Jianfeng} for a discussion on a more general setting where the dynamics of $X$ and $P$ are allowed to be SDE. However, this algorithm  has a major drawback, that is, the optimization on the top line is hard to solve in high dimension (unless in some special cases when the optimizers have analytic forms). 
That is why after the discovery of Pontryagin's maximum principle, people have studied and widely applied the related gradient descent algorithm, see e.g. \cite{BD62, Mit66}. In such iterative algorithm at each step we update the value of the control variable along the direction opposite to that of $\nabla_a H$, that is,
\beaa
\a^{i+1} := \a^i - \eta \nabla_a H(t, X^{i}_t, a, P^{i}_t),
\eeaa
where $\eta$ is the learning rate and $X^{i}, P^{i}$ are the forward and backward processes evaluated with $\a^i$. In order to look into the convergence of such iteration, let us consider the continuous version of this gradient descent algorithm, governed by the following system of ODEs on the infinite horizon:
\bea\label{sys_ODE}
\left\{
\begin{array}{ll}
\frac{d\a^s_t}{ds} =  - \nabla_a H(t, X^s_t, a^s_t , P^s_t) \q\mbox{on}~~\{s\ge 0\}\q\mbox{for all}~~t\in[0,T],\\
\\
X^s_t =x_0 + \int_0^t  \phi(r, X^s_r, \a^s_r)dr, \\
P^s_t  =  \nabla_x G(X^s_T) + \int_0^t \nabla_x H(r, X^s_r, \a^s_r, P^s_r) dr.\\
\end{array}
\right.
\eea
Curiously, after some careful calculus, one may verify that
\beaa
\frac{dV^0(\a^s)}{ds} = - \int_0^T \Big|  \nabla_a H(r, X^s_r,a^s_r, P^s_r ) \Big|^2 dr.
\eeaa
Therefore $V^0$ is a natural Lyapunov function for the process $(\a^s)$, and in order for the equality $\frac{dV^0(\a)}{ds} =0$ to be true, the control $\a$ must satisfy the forward-backward system \eqref{eq:MP}.  This analysis (though not completely rigorous) reflects why this  algorithm would converge. However, like other gradient-descent type algorithms, it would converge to a local minimizer, since 
Pontryagin's maximum principle is only a necessary first-order condition. One may attempt to put a convexity condition on the coefficients in order to ensure the local minimizer to be the global one. However, this usually urges $X$ to be linear in $\a$ (so the function $\phi$ needs to be linear in $(x,a)$), which largely limits the application of this method. 

In order to go beyond the convex case for the optimal control problem, it is natural to recall how the Langevin equation helps to approximate the solution of the non-convex optimization on the real space. Given a function $F$ not necessarily convex, we know that under some mild conditions the unique invariant measure of the following Langevin equation
\bea\label{eq:classicLangevin}
d\Th_s = -\dot F (\Th_s) ds + \si d W_s
\eea
is the global minimizer of the regularized optimization:
\bea\label{nonconvexopt}
\min_{\nu\in \cP } \int_{\dbR^m} F(a) \nu(da) + \frac{\si^2}{2} \Ent(\nu),  
\eea
where $ W $ is the Brownian motion, $\cP$ is the space of probability measures and the regularizer  $\Ent$ is the relative entropy with respect to the Lebesgue measure, see e.g. \cite{JKO98}. Moreover, the marginal law of the process \eqref{eq:classicLangevin} converges to its invariant measure.  As analyzed in the recent paper \cite{HRSS19}, this result is basically due to the fact that the function $\nu\mapsto  \int F(a) \nu(da) $ is convex (indeed linear). 
In the present paper we wish to apply a similar regularization to the optimal control problem. In order to do that we first recall the relaxed formulation of the control problem \eqref{standardcontrol}. Instead of controlling the process $\a$, we will control the flow of laws $(\nu_t)_{t\in [0,T]}$. Then the controlled process reads
\beaa
X_t  = x_0 + \int_0^t \int_{\dbR^m} \phi(r, X_r, a) \nu_r(da ) dr,
\eeaa
and we aim at minimizing 
\beaa
\inf_\nu V(\nu), \q\mbox{where}\q V(\nu):= \int_0^T \int L(t, X_t , a)\nu_t(da) dt +G(X_T).
\eeaa
Comparing it to the original control problem \eqref{standardcontrol}, we obtain that $\inf_\nu V(\nu) \le \inf_\a V^0(\a)$. Indeed, due to the classical results in \cite{Fle77, ElNJ87}, under some mild conditions the values of the minimums of the two formulations remain the same. Further we add the relative entropy as a regularizer, as in \eqref{nonconvexopt}, and focus on the regularized optimization:
\bea\label{intro:reg_control}
\inf_\nu V^\si(\nu), \q\mbox{where}\q V^\si(\nu):= V(\nu) + \frac{\si^2}{2}\int_0^T \Ent(\nu_t)dt.
\eea
We will show that the global minimizer of this regularized control problem is again characterized  by the invariant measure of Langevin-type dynamics, however, not a single Langevin equation as in \eqref{eq:classicLangevin}, but a system of mean-field Langevin equations in the spirit of \eqref{sys_ODE}, that is,
\bea\label{intro:sysLangevin}
\left\{
\begin{array}{ll}
	d \Th_t^s = - \nabla_a H(t, X_t^s, \Th_t^s, P_t^s)  ds + \si dW_s, \q\q\q\mbox{for $s\in \dbR^+$},\q \mbox{for  $t\in [0,T]$,} \q\mbox{where}\\
	\\
X_t^s = X_0 + \int_0^t \int_{\dbR^m} \phi(r,X_r^s, a) \nu_r^s (da) dr, ~\q\q\q\mbox{with}\q \nu_r^s := \mbox{Law}(\Th_r^s),\\
P_t^s = \nabla_x G(X^s_T) + \int_t^T \int_{\dbR^m} \nabla_x H(r, X_r^s, a, P_t^s)  \nu_r^s(da)   dr,
\end{array}
\right.
\eea
The name `mean-field' reflects the fact that the different equations in the system are coupled through (and only through) the marginal laws $(\nu^s_t)_{t\in [0,T], s\in \dbR+}$.
 Moreover, we shall show that this characterization  holds true not only when $V$ is convex in $\nu$ (which is still a quite restrictive case), but also under a set of milder conditions on the coefficients. Also, we prove in both cases that the marginal laws of the system \eqref{intro:sysLangevin} converge to its unique invariant measure.  In particular, in the latter case we may quantitively compute the convergence rate. 

One  concrete motivation of this work is to shed some light on the solvability of the gradient descent method for the deep neural networks.  Our work can be viewed as a  natural extension to the recent works \cite{mei2018mean, mei2019mean, HRSS19} in which the authors endorse the solvability of the two-layer (i.e. with one hidden layer) neural networks using the mean-field Langevin equations. 
It has been  proposed  in the recent papers \cite{CLT19, CRBD19, EHL19, LM19} among others,  as well as  in the course of P.-L. Lions in Coll\`ege de France (indeed similar ideas can be dated back to \cite{LTHS88, Pea95}, see also the very recent review on this topic \cite{LT19}),  that one may use the continuous-time optimal control problem  as a model to study the deep neural networks. However, to our knowledge, there is no existing literature which succeeds in explaining why the stochastic gradient descent algorithm may approach the global optimum of the deep neural network under mild conditions.  Our system of mean-field Langevin equations \eqref{intro:sysLangevin} and its relation to the regularized optimization \eqref{intro:reg_control} show a clear clue to how numerically compute the optimal control. Meanwhile, it is curious to observe that the standard discretization scheme (explicite Euler scheme) for the  dynamics \eqref{intro:sysLangevin} is equivalent to the (noised) stochastic gradient descent algorithm for a class of deep neural networks, such as residual networks, convolutional networks, recurrent networks and so on.  

The rest of the paper is organised in the following way. In Section \ref{sec:prelim}, we define the relaxed optimal control problem under study and the corresponding system of mean-field Langevin equations. In Section \ref{sec:mainresult} we announce the main results of the paper, namely, the wellposedness of the system of mean-field Langevin equations and  the convergence of the marginal laws of the system towards the optimal control, both in the convex case and in a contraction case. Before giving detailed proofs for the theoretical results, we introduce the application to deep neural networks in Section \ref{sec:application}. Then in Sections \ref{sec:proofWell}, \ref{sec:convex} and \ref{sec:contraction} we provide the proofs of the main results.

\section{Preliminaries}\label{sec:prelim}

\subsection{Regularized Relaxed Optimal Control} 

In this paper we aim to solve the optimal control problems in large dimension (in particular, the control variable is of large dimension). We shall allow the player to apply a mixed strategy, namely, a probability measure $\nu$ of which the marginal law on the time dimension is  the Lebesgue measure, i.e.
\beaa
\nu \in \cV:= \Big\{ \nu \in \cM([0,T]\times \dbR^m) :\q  \nu(dt, da) = \nu_t(da)dt, \q\mbox{for some}~~ \nu_t \in \cP(\dbR^m) \Big\},
\eeaa
where we denote by $\cM$ the space of measures. 
The controlled process $X$ reads:
\beaa
X_t = X_0 + \int_{0}^{t}\int_{\dbR^m} \phi(r,X_r, a,Z)\nu_r(\mathrm{d}a)\mathrm{d}r, \q\mbox{for}\q t\in [0,T],
\eeaa
where $Z$ is an exogenous random variable taking values in a set $\cZ$. In particular, in the application to the neural networks, $Z$ would represent the input data. Denote $\cG:=\si(Z)$ and assume that $X_0$ is a bounded  $\cG$-measurable random variable.
%
%
We use the notation $\E$ as the expectation of random variables on $\cG$. The relaxed control problem writes:
\begin{equation}\label{control_problem}
\inf_{\nu}V(\nu),
\q\mbox{where}\q V(\nu) := \E\left[\int_{0}^{T}\int_{\dbR^m} L(t ,X_t, a, Z)\nu_t(\mathrm{d}a)\mathrm{d}t + G(X_T, Z)\right].
\end{equation}
Further in this paper, instead of addressing the optimal control problem itself, we introduce the following regularized version:
\bea\label{reg_control}
\inf_{\nu}V^\si(\nu),
\q\mbox{where}\q V^\si(\nu) :=  V(\nu) + \frac{\si^2}{2} \int_0^T \Ent(\nu_t) dt,
\eea
where $\Ent$ is the relative entropy with respect to the Lebesgue measure on $\dbR^m$.  It is noteworthy that $ \int_0^T \Ent(\nu_t) dt$ is equal to the relative entropy of $\nu$ with respect to the Lebesgue measure on $[0,T]\times \dbR^m$.

\subsection{System of Mean-Field Langevin Equations}

The following remark establishes a link between the control problem (\ref{reg_control}) and the mean-field Langevin equation.

\begin{rem}\label{rem:onelayer}
Let us consider a simple example of a control problem with the following coefficients: $X_0\equiv 0$, $L(a)=\l |a|^2$ and $\phi(t,x, a) = \hat\phi(a)$, that is, we aim to minimize
\beaa
\inf_\nu ~ \E\Big[G\Big( \int_0^T\int_{\dbR^m} \hat\phi(a, Z)\nu_t(da)dt\Big) \Big] +\int_0^T \int_{\dbR^m} \l |a|^2 \nu_t(da)dt+ \frac{\si^2}{2} \int_0^T \Ent(\nu_t) dt.
\eeaa
Clearly, $(\nu_t)_{t\in [0,T]}$ are exchangeable, so the optimal control $\nu^*$ must satisfy $\nu^*_0= \nu^*_t$ for any $t\in [0,T]$. Therefore it is equivalent to minimize
\beaa
\inf_{\nu_0} ~ \E\Big[ G\Big(T \int_{\dbR^m} \hat\phi(a, Z)\nu_0(da)\Big)\Big]+T\int_{\dbR^m} \l |a|^2 \nu_0(da) + \frac{\si^2T}{2}  \Ent(\nu_0).
\eeaa
Given a convex function $G$, this minimization problem is studied in the recent paper \cite{HRSS19}, where the authors prove that the marginal laws of the corresponding mean-field Langevin equation converge to the global minimizer. In the present paper we are going to generalize this result.
\end{rem}

For the general control problem \eqref{reg_control}, we assume that all the coefficients are smooth enough. Let $(\O, \cF, \dbP)$ be a probability space and $W$ an $m$-dimensional Brownian motion on it. Introduce the following system of mean-field Langevin equations:
\bea\label{sys_Langevin}
\left\{
\begin{array}{ll}
	d \Th_t^s = - \E\big[\nabla_a H(t, X_t^s, \Th_t^s, P_t^s, Z)\big]  ds + \si dW_s, \q\q\q\mbox{for $s\in \dbR^+$},\q \mbox{for  $t\in [0,T]$,}\\
	\\
	\mbox{where}\q X_t^s = X_0 + \int_0^t \int_{\dbR^m} \phi(r,X_r^s, a, Z) \nu_r^s (da) dr, ~\q\q\q\mbox{with}\q \nu_r^s := \mbox{Law}(\Th_r^s),\\
	\q\q\q~~ P_t^s = \nabla_x G(X^s_T, Z) + \int_t^T \int_{\dbR^m} \nabla_x H(r, X_r^s, a, P_r^s, Z)  \nu_r^s(da)   dr,
\end{array}
\right.
\eea
and  $H$ is the Hamiltonian function:
\beaa
H(t,x,a, p, z) := L(t,x, a, z) + p \cd \phi(t,x,a,z)\q\q\mbox{for}\q (t,x,a,p)\in [0,T]\times \dbR^d\times \dbR^m \times \dbR^d.
\eeaa
For readers who are familiar with variational calculus of optimal control (Pontryagin's maximum principle), we note that the process $(P_t^s)_{t\in [0,T], s\in\dbR^+}$ has an obvious link to the adjoint process in the maximum principle. This connection will be made clear in the discussion of Section \ref{sec:convex}.  We are going to prove that under reasonable assumptions the system of mean-field Langevin equations has a unique solution, and the marginal distribution $(\nu_t^s)_{t\in [0,T]}$ converges to the global minimizer of the control problem \eqref{reg_control}  as $s\rightarrow\infty$.

\subsection{Notation}

Let $(\O, \cF, \dbP)$ be a probability space where lives the Brownian motion $ W $. Denote by $\dbE$ the expectation under the probability $\dbP$, or equivalently, the expectation of the randomness produced by the Brownian motion $W$. In particular, note the difference between the notations $\dbE$ and $\E$.

In the present paper we shall use several different metrics on the measure space. First, recall the $p$-Wasserstein distance ($p\ge 1$) on the probability space $\cP(\dbR^m)$:
\beaa
\cW_p(\mu_0, \nu_0)^p := \inf\left\{ \int_{\dbR^m} |x-y|^p \pi(dx,dy): \q\mbox{where $\pi$ is a coupling of $\mu_0,\nu_0\in \cP(\dbR^m)$}  \right\}. 
\eeaa
Further, for $\mu,\nu\in \cV$ we define the metric 
\beaa
\ol\cW^T_{p}(\mu,\nu) : = \Big(\int_0^T \cW_p (\mu_t, \nu_t)^p dt\Big)^{1/p}.
\eeaa
 In some part of the paper, in particular during the discussion of the convex case (see Section \ref{subsec:convex} and \ref{sec:convex}), we shall use the following generalized $p$-Wasserstein distance on $\cV$:
 \bea\label{eq:Wdist_prodspace}
  \cW_p^T(\mu, \nu) := T^{1/p} \cW_p\Big(\frac{\mu}{T}, \frac{\nu}{T}\Big)\q\mbox{for $\mu,\nu\in \cV$}.
 \eea
Comparing the above definitions, clearly we have $\cW^T_p(\mu,\nu) \le \ol \cW^T_p (\mu,\nu)$ for any $\mu,\nu\in \cV$.


In the proofs of the following sections, the constant $C$ may vary from line to line. Without further specification, $C$ is always strictly positive.

\section{Main Results}\label{sec:mainresult}

In this section we announce the main results. Their proofs are given in Sections \ref{sec:proofWell}, \ref{sec:convex} and \ref{sec:contraction}.

Throughout the paper we assume that  the Hamiltonian function $H$ and the terminal cost $G$ are continuously differentiable in the variables $(x,a)$, and the coefficients 
\beaa
\phi, \nabla_x G, \nabla_x L, \nabla_x \phi\q\mbox{exist and are all bounded}.
\eeaa
Therefore $(X_t^s, P_t^s)$ lives in a compact set $\cK_x\times\cK_p$. From now on we treat  $(t,x,a,p,z)\mapsto H(t,x,a,p,z)$ as a function defined on $[0,T]\times \cK_x\times \dbR^m \times \cK_p\times \cZ$. In particular, whenever we claim that $H$ satisfies a property (e.g. Lipschitz continuity) globally, it is meant to be true on this set instead of the whole space.

\subsection{Wellposedness of the System of Mean-Field Langevin Equation}

\begin{assum}\label{assum:well}
Assume that  the coefficients $\phi, L, G$ are continuously differentiable in the variables $(x,a)$ and 
\begin{itemize}
\item  $\nabla_a \phi, \nabla_a L, \nabla_x \phi, \nabla_x L,  \nabla_x G, \phi$ are  uniformly Lipschitz continuous in the variables $(x, a, p)$;

\item the coefficients  $H, \nabla_a H$ satisfy
\bea
\sup_{t,z}|H(t,0,0,0, z)|<\infty, \q \sup_{t,z} | \nabla_a H(t,0,0,0,z)| < \infty. \label{assum:uniformLG}
\eea
\end{itemize}
\end{assum}

\no Define the space of the continuous measure flows on the horizon $[0,S]$:
\beaa
C_p\big([0,S], \cV \big):=\left\{ \mu=(\mu^s)_{s\in [0,S]} : \q  \mu^s\in \cV 
				~~\mbox{and}~~ \lim_{s'\rightarrow s} \ol\cW_p^T(\mu^{s'}, \mu^s) =0~~\mbox{for all}~~ s\in 					[0,S]\right\}. 
\eeaa

\begin{thm}\label{thm:well}
	Let Assumption \ref{assum:well} hold true. Given $(\Th^0_t)_{t\in [0,T]}$ such that
	\bea\label{assum:Th0}
	\int_0^T \dbE[ |\Th^0_t|^{p} ] dt <\infty, \q\q\mbox{for some}\q p\ge 1,
	\eea
	the  system of SDE \eqref{sys_Langevin} has a unique solution. In particular, the law of the solution $(\nu^s)\in C_p\big(\dbR^+, \cV \big)$.
\end{thm}

One of our main contributions is to observe the decrease of energy along the flow of the solution to the system of mean-field Langevin equations \eqref{sys_Langevin}.

\begin{assum}\label{assum:GF}
We further assume that
\begin{itemize}
\item the coefficients $\phi, L$ are second-order continuously differentiable in $a$;

\item  there is $\e>0$ such that  
\bea
 a \cd \nabla_a H(t,x,a, p,z) \ge  \e |a|^2, \q\mbox{for $|a|$  big enough} ; \label{assum:dissipative}
 \eea

\item for fixed $(t,x ,p)$ the mapping $a\mapsto \E\big[\nabla_a H(t,x,a,p,Z)\big]$ belongs to $C^3$.

\end{itemize}
\end{assum}

\begin{thm}[Gradient flow]\label{thm:gradient_flow}
Let Assumptions \ref{assum:well} and \ref{assum:GF} hold true, and assume that
\bea\label{assum:Th0p2}
 \int_0^T \dbE\big[ |\Th^0_t|^p \big] dt <\infty \q\mbox{for some} \q p\ge 2.
\eea
\no Recall the function $V^\si$ defined in \eqref{reg_control}. Let $(\nu^s_t)$ be the marginal laws of the solution to the system of mean-field Langevin equations \eqref{sys_Langevin}. Then, for each $s>0$, $t\in [0,T]$ the law $\nu^s_t$ admits a density, and  for $s'>s>0$ we have
\bea\label{eq:gradient_flow}
V^\si(\nu^{s'}) - V^\si(\nu^s)
=
- \int_s^{s'}\int_0^T\int_{\dbR^m} \Big| \E\big[\nabla_aH(t,X^r_t,a,P^r_t,Z)\big] + \frac{\si^2}{2}\nabla_a\ln \nu^r_t(a)\Big|^2\nu^r_t(a) d a d t  dr.
\eea
\end{thm}

\ms

\subsection{Convex Case}\label{subsec:convex}

\no  We first consider the case where the objective function $V$, defined in \eqref{control_problem},  is convex in $\nu$. More precisely, we assume the following.

\begin{assum}\label{assum:convex}
Let the controlled process $X$ be linear in $\nu$, i.e.
\bea\label{eq:controlproc_convex}
	X_t = X_0 + \int_{0}^{t}\int_{\dbR^m} \phi(s, a, Z)\nu_s( d a) d s,
\eea
and the dependence on the variables $(x,a)$ of the function $L$ is seperated, i.e. 
\beaa
L(t,x,a,z) = \ell (t,x,z) + c(t,a,z).
\eeaa 
Further we assume that
\begin{itemize}
\item for all $ t\in[0,T] $ the functions $ \ell, G $ are convex in $x$;
		
		\item the functions $ \phi, \nabla_x H , \nabla_a H $ are globally Lipschitz continuous  in $ t $;
		
		\item $H$ is continuously differentiable in $t$ and $\pa_t H$ is globally Lipschitz continuous in $(t,a)$.
		\end{itemize}
\end{assum}

\begin{rem}
In the present section concerning the convex case we add the regularity assumptions on the coefficients with respect to the variable $t$. That is due to the fact that in this part of the paper we will apply the metric $\cW_p^T$ (defined in \eqref{eq:Wdist_prodspace}) on the space $\cV$ instead of the usual one $\ol\cW_p^T$. 
\end{rem}

\no Under the above assumptions, it is clear that there exists at least  one global minimizer of $V^\si$. Moreover,  the function $V$ is convex in $\nu$, and thus $V^\si$ is strictly convex in $\nu$ for any $\si>0$, so there is one unique  global minimizer. By standard variational calculus, the following theorem states a sufficient condition for a control to be the unique global minimizer of the our problem \eqref{reg_control}. 

\begin{thm}[Sufficient first order condition]\label{thm:FOC}
	Let Assumption \ref{assum:convex} holds true. If  $ \nu^*\in\cV$, equivalent to the Lebesgue measure, satisfies
	\begin{equation}\label{first_order_condition}
		\E \big[\nabla_a H(t, X^*_t, \cd, P^*_t, Z) \big] +\frac{\si^2}{2} \nabla_a \ln\big(\nu_t^*\big) = 0
	\end{equation}
	for Leb-a.s. $t$, where $X^*$ is the controlled process with the control $\nu^*$ as in \eqref{eq:controlproc_convex} and $P^*$ is the following adjoint process
	\bea\label{eq:adjointproc}
	P^*_t := \nabla_x G(X^*_T, Z) + \int_{t}^{T}\nabla_x\ell(r,X^*_r,Z)  d r,
	\eea
	then $ \nu^* $ is an optimal control of the regularized control problem \eqref{reg_control}.
\end{thm}

Combining the sufficient condition above and Theorem \ref{thm:gradient_flow}, we can prove the following main result in the convex case.

\begin{thm}\label{thm:convex_convergence}
Let Assumptions \ref{assum:well}, \ref{assum:GF} and \ref{assum:convex} hold true and let $ (\Th^0_t)_{t\in[0,T]} $ satisfy \eqref{assum:Th0p2} with $p>2$. Further assume that $V$ is $ \cW_2^T$-continuous and bounded from below. Denote $ (\nu^s_t)^{s\in\dbR^+}_{t\in[0,T]} $ the flow of marginal laws of the solution to  \eqref{sys_Langevin}. Then there exists an invariant measure of \eqref{sys_Langevin} equal to $ (\nu^*_t)_{t\in[0,T]}:=\argmin\limits_{\nu}V^\si(\nu) $, and $ (\nu^s_t)_{t\in[0,T]} $ converges to $ (\nu^*_t)_{t\in[0,T]} $.
\end{thm}

\subsection{Contraction Case}\label{subsec:contraction}
	
Clearly, the previous convex case has restrictive requirements on the structure of the coefficients. In particular, these requirements cannot all be satisfied in the application to deep neural networks. That drives us to look for another setting in which the system of mean-field Langevin equations leads us to the optimal control of \eqref{reg_control}.

\begin{prop}[Necessary first order condition]\label{prop:NC}
Let Assumptions \ref{assum:well}, \ref{assum:GF} hold true and  assume that the function $V$ is bounded from below. Let $ \nu^* $ be an optimal control of $ V^\si $ such that $V^\si (\nu^*)<\infty$. Then $ \nu^* $ is equivalent to the Lebesgue measure and satisfies
\bea\label{eq:NC}
		\E \big[\nabla_a H(t, X^*_t, \cd, P^*_t, Z) \big] +\frac{\si^2}{2} \nabla_a \ln\big(\nu_t^*\big) = 0, \q\mbox{for Leb-a.s. $t\in[0,T]$,}
\eea
where $ (X^*,P^*) $ is the solution of the following ODE
\beaa
\left.
\begin{array}{ll}
	& X_t^* = X_0 + \int_0^t \int_{\dbR^m} \phi(r,X_r^*, a, Z) \nu_r^* (da) dr,\\
	& P_t^* = \nabla_x G(X^*_T, Z) + \int_t^T \int_{\dbR^m} \nabla_x H(r, X_r^*, a, P_t^*, Z)  \nu_r^*(da)   dr.
\end{array}
\right.
\eeaa
In particular, $\nu^*$ is an invariant measure of the system \eqref{sys_Langevin}.
\end{prop}


 Assume that the control problem \eqref{reg_control} admits at least one optimal control. The proposition  above  implies that once we ensure the convergence of the marginal laws of the system \eqref{sys_Langevin} towards the unique invariant measure, then the limit measure is an (indeed the unique) optimal control of \eqref{reg_control}.

Next we find a sufficient condition for the existence of the unique invariant measure for the system of mean-field Langevin equations. In particular, the convergent rate towards the limit measure is computed explicitly.

\begin{assum}\label{assum:contraction}
Assume that  there is a continuous function $ \kappa:(0,+\infty)\rightarrow\dbR $ such that  $ \limsup\limits_{r\rightarrow+\infty}\kappa(r)<0$, $\int_{0}^{1}r\kappa(r) d r<+\infty  $ and for any $(t,x,p,z)$ we have
\begin{equation*}
(a - \tilde a)\cdot\Big(-\nabla_a H(t,x,a,p,z) + \nabla_a H(t,x,\tilde a,p,z)\Big) \leq \kappa\left(|a - \tilde a|\right)\left|a - \tilde a\right|^2\q\text{for all }a,\tilde a\in\dbR^m,~ a\neq\tilde a.
\end{equation*}
\end{assum}

\begin{thm}[Contraction towards the unique invariant measure]\label{thm:contraction}
Let Assumptions \ref{assum:well},  \ref{assum:GF} and  \ref{assum:contraction} hold true. Let $ (\Th^0_t)_{t\in[0,T]} $, resp. $ (\tilde \Th^0_t)_{t\in[0,T]} $, satisfy \eqref{assum:Th0} and denote $ \nu^s $, resp. $ \tilde \nu^s $, the marginal law of the solution of the mean-field Langevin system \eqref{sys_Langevin}. Then there exist 
 constants $ c, \g\in(0,\infty) $ such that for any $ s\geq 0 $, we have
\begin{equation}\label{eq:contractionresult}
	 \ol\cW^T_1 (\nu^s, \tilde\nu^s) 
	\le e^{\left(\frac{2\g T}{\f(R_1)}-c\sigma^2\right)s}  \frac{2}{\f(R_1)}  \ol\cW^T_1 (\nu^0, \tilde\nu^0) .
\end{equation} 
The constants are explicitly given by
\beaa
	\f(r) = \exp\left( -\frac12\int_{0}^{r}\frac{u\kappa^+(u)}{\sigma^2} d u \right),\q
	c^{-1} = \int_{0}^{R_2}\Phi(s)\f(s)^{-1} d s\q\text{and}\q \g = K^2(1+K)\exp(2KT),
\eeaa
where $ K $ is a common Lipschitz coefficient of $ \nabla_aH $, $ \nabla_xH $, $ \nabla_xG $ and $ \phi $, $\Phi(r) = \int_{0}^{r}\f(s) d s$, and 
\begin{equation*}
R_1 := \inf\{ R\geq 0:\kappa(r)\leq0 \text{ for all }r\geq R \},\q 
R_2 := \inf\{ R\geq R_1:\kappa(r)R(R-R_1)\leq-4\si^2\text{ for all }r\geq R \}.
\end{equation*}
In particular, if $\frac{2\g T}{\f(R_1)} < c\sigma^2$, there is a unique invariant measure with finite $1$-moment.
\end{thm}

\begin{rem}
The result and the proof of Theorem \ref{thm:contraction} reveal the importance of considering the relaxed formulation of the control problem instead of the strict one \eqref{standardcontrol}. As discussed in the introduction, in the setting of the strict formulation, one may let the control $(\a_t)_{t\in[0,T]}$ evolve along the gradient as in \eqref{sys_ODE}. In this case the limit  $\lim\limits_{s\rightarrow\infty}(a^s_t)_{t\in [0,T]}$, if exists, depends in general on the initial value $(\a^0_t)_{t\in [0,T]}$, so it is unlikely to be the optimal control. On the contrary, Theorem \ref{thm:contraction} ensures that the gradient flow \eqref{sys_Langevin} in the relaxed formulation converges to the unique invariant measure independent of the initialization $\nu^0$. 
\end{rem}

In order to link the unique invariant measure to the optimal control of \eqref{reg_control}, we need an additional assumption.

\begin{assum}\label{assum:V}
Assume that 
\begin{itemize}
\item the functional $V$ is  bounded from below,
\item $V^\si$ has at least an optimal control $\nu^*$ such that $V^\si(\nu^*)<\infty$.
\end{itemize}
\end{assum}
\no In particular, we will see a sufficient condition for the existence of  optimal control  in Lemma \ref{lem:exist_optimal}. Finally, combining the results in Proposition \ref{prop:NC} and Theorem \ref{thm:contraction}, we may conclude:

\begin{cor}\label{cor:converge}
Let Assumptions  \ref{assum:well}, \ref{assum:GF},  \ref{assum:contraction} and \ref{assum:V} hold true. Recall the constants defined in Theorem \ref{thm:contraction}. If $\frac{2\g T}{\f(R_1)} < c\sigma^2$, then the unique invariant measure of \eqref{sys_Langevin} is the optimal control of \eqref{reg_control}.
\end{cor}


\section{Application to Deep Neural Networks}\label{sec:application}

In this section, we apply the previous theoretical results, in particular the results of Section \ref{subsec:contraction}, to show the solvability of the stochastic gradient descent algorithm for optimizing the weights of a deep neural network. 

Let the data $Z$ take value in a compact subset $\cZ$ of $\dbR^D$, and denote by $Y=f(Z)$ the label of the data. The function $f$ is unknown, and we want to approximate it using the parametrized function generated by a deep neural network. Here we shall model the deep neural network using a controlled dynamic. 

More precisely, consider the following choice of coefficients of the control problem \eqref{reg_control}: for $x\in \dbR^d, \b \in\dbR^d\times \dbR^d, A\in \dbR^d\times\dbR^d, \tilde A\in\dbR^d\times\dbR^D, k\in \dbR^d,  a=(\b, A, \tilde A, k)\in \dbR^m, z\in \cZ$
\bea\label{eq:DNNcoeff}
\phi(t, x,a,z):= \ell(\b) \f\big( t, \ell(A)x+ \tilde A z+k\big), 
\q L(a) := \l |a|^2,
\q G(x, z) :=  g\big(Tf(z) - x \big)
\eea
where $\l>0$ is a constant, $\f$ is a nonlinear activation function, $\ell$ is a bounded truncation function and $g$ is  a cost function.

Assume that the coefficients satisfy all the assumptions needed for Corollary \ref{cor:converge}. 
Recall that the terminal value of the controlled process is equal to 
\beaa
X_T = X_0 + \int_0^T\int_{\dbR^m} \phi(t, X_t, a, Z) \nu_t(da)dt,
\eeaa
so it is a parametrized function of $Z$, where $(\nu_t)_{t\in [0,T]}$ is the parameter process. As we solve the regularized relaxed optimal control problem \eqref{reg_control}, we find the optimal parameter $\nu$ so as to minimize the (regularized) statistic error between the label $Y$ and the output $\frac1T X_T$.  Once we discretize the equation using the explicit Euler scheme and introduce random variables  $(\Th^j_t)_{j=1, \cds, n_t}$, which are independent copies following the law $\nu_t$, we find
\bea
\label{forw_prop}
X_{t_{i+1}} \approx X_{t_{i}} + \frac{ \d t}{n_{t_{i+1}}} \sum_{j=1}^{n_{t_{i+1}}} \phi(t_i, X_{t_i}, \Th^j_{t_{i+1}}, Z),\q\mbox{where $\d t:= t_{i+1}-t_i$}.
\eea
This discrete dynamics characterizes a type of  structure of deep neural networks, and can be equivalently represented by the scheme of the figure below. 

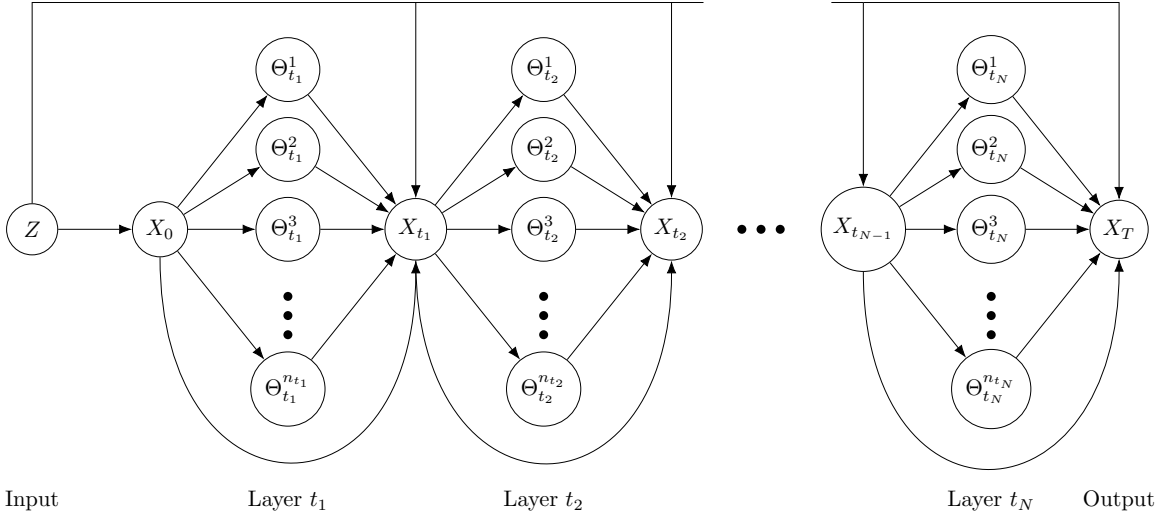
\begin{figure}[h]
\scalebox{0.85}{
    \begin{tikzpicture}
    	\tikzstyle{place}=[circle, draw=black, minimum size = 8mm]
    	
	\draw node at (0, -3*1.25) [place] (z) {$Z$};    
    	\draw node at (2, -3*1.25) [place] (x_0) {$X_{0}$};
    	
    	\foreach \i in {1,...,3}
    		\draw node at (4, -\i*1.25) [place] (first_\i) {$\Theta^\i_{t_1}$};
    	\foreach \i in {1,...,3}
    		\fill (4, -4.5 -\i*0.3) circle (2pt);
    	\draw node at (4, -5*1.25) [place] (first_n) {$\Theta^{n_{t_1}}_{t_1}$};
    	
    	\draw node at (6, -3*1.25) [place] (x_1) {$X_{t_1}$};
    	
    	\foreach \i in {1,...,3}
    		\node at (8, -\i*1.25) [place] (second_\i){$\Theta^\i_{t_2}$};
    	\foreach \i in {1,...,3}
    		\fill (8, -4.5 -\i*0.3) circle (2pt);
    	\draw node at (8, -5*1.25) [place] (second_n) {$\Theta^{n_{t_2}}_{t_2}$};
    	
    	\draw node at (10, -3*1.25) [place] (x_2) {$X_{t_2}$};
    	
    	\foreach \i in {1,...,3}
    		\fill (12 - \i*0.3, -3*1.25 ) circle (2pt);
    		
    	\draw node at (13, -3*1.25) [place] (x_last) { $X_{t_{N-1}}$};
    	
	\foreach \i in {1,...,3}
    		\node at (15, -\i*1.25) [place] (last_\i){$\Theta^\i_{t_N}$};
	\foreach \i in {1,...,3}
    		\fill (15, -4.5 -\i*0.3) circle (2pt);
    	\node at (15, -5*1.25) [place] (last_n) {$\Theta^{n_{t_{N}}}_{t_N}$};
    	
    	\draw node at (17, -3*1.25) [place] (x_T) {$X_{T}$};
    		
	\draw [-{Latex[length=2mm]}] (z) to (x_0);    
		
    	\foreach \i in {1,...,3}
    		\draw [-{Latex[length=2mm]}] (x_0) to (first_\i);
	\draw [-{Latex[length=2mm]}] (x_0) to (first_n);
	
    	\foreach \i in {1,...,3}
    		\draw [-{Latex[length=2mm]}] (first_\i) to (x_1);
	\draw [-{Latex[length=2mm]}] (first_n) to (x_1);
	
    	\foreach \i in {1,...,3}
    		\draw [-{Latex[length=2mm]}] (x_1) to (second_\i);
	\draw [-{Latex[length=2mm]}] (x_1) to (second_n);
    	
    	\foreach \i in {1,...,3}
    		\draw [-{Latex[length=2mm]}] (second_\i) to (x_2);
	\draw [-{Latex[length=2mm]}] (second_n) to (x_2);
	
    	\foreach \i in {1,...,3}
    		\draw [-{Latex[length=2mm]}] (x_last) to (last_\i);
	\draw [-{Latex[length=2mm]}] (x_last) to (last_n);
	
    	\foreach \i in {1,...,3}
    		\draw [-{Latex[length=2mm]}] (last_\i) to (x_T);
	\draw [-{Latex[length=2mm]}] (last_n) to (x_T);
	
  	\draw [-{Latex[length=2mm]},looseness=2.7] (x_0) to [out=270,in=270] (x_1);
  	\draw [-{Latex[length=2mm]},looseness=2.7] (x_1) to [out=270,in=270] (x_2);
  	\draw [-{Latex[length=2mm]},looseness=2.7] (x_last) to [out=270,in=270] (x_T);
  	
  	\draw [-{Latex[length=2mm]}] (z.north) -- (0,-0.2)-| (x_1);
  	\draw [-{Latex[length=2mm]}] (6,-0.2) -| (x_2);
  	\draw  (10,-0.2) -- (10.5,-0.2);
  	\draw  (12.5,-0.2) -- (13,-0.2);
  	\draw  [-{Latex[length=2mm]}] (13,-0.2) -- (x_last.north);
  	\draw [-{Latex[length=2mm]}] (13,-0.2) -| (x_T);
    	
    	\node at (0, -8) [black, ] {Input};
    	\node at (4, -8) [black, ] {Layer $t_1$};
    	\node at (8, -8) [black, ] {Layer $t_2$};
    	\node at (15, -8) [black, ] {Layer $t_N$};
    	\node at (17, -8) [black, ] {Output};
    \end{tikzpicture}
}
\caption{Neural network corresponding to the relaxed controlled process}
\label{fig:NNscheme}    
\end{figure}

    \no This structure can describe a class of widely (and successfully) applied deep neural networks. Here are some examples:
    \begin{itemize}
    \item   the process $X$ can be viewed as the outputs of intermediate layers in a Residual Neural Network, see \cite{He_2016_CVPR};
    \item   $(X_{t_i})$ can also be interpreted as recurrent neurones in a Recurrent Neural Network or an LSTM Neural Network, see \cite{LSTM97};
    \item   once we take the function $\f$ to be a convolutional-type activation function, the structure forms a Convolutional Neural Network with average pooling, see \cite{NIPS2012_4824}.
    \end{itemize}
The most significant feature of this structure is the average pooling (averaging the outputs of the nonlinear activation $\f$ as in  \eqref{forw_prop}) on each layer, and it is due to the adoption of the relaxed formulation in our model of controlled process.

Given the structure of the neural network, or the scheme of the forward propagation (\ref{forw_prop}), it is  conventional  to optimize the parameters $(\Th^j_{t_i})$ using the stochastic gradient descent algorithm. The gradients of the parameters are easy to compute, due to the chain rule (or backward propagation):
\beaa
\left.
\begin{aligned}
& \Th_t^{s_{j+1}} =  \Th_t^{s_j} -  \d s \E\big[\nabla_a H(t, X_t^{s_j}, \Th_t^{s_j}, P_t^{s_j}, Z ) \big] + \si \d W_{s_j}, \text{ with } \d s = s_{ j+1 } -s_j, \\
& \text{where } P_{t_{i-1}}^s = P_{t_i}^s - \d t \sum_{j=1}^{n_{t_{i+1}}}  \nabla_x H\big(t_i, X_{t_{i}}^s, \Th^{j,s}_{t_i}, P_{t_i}^s, Z\big), \q P_{ T }^s = \nabla_x G(X^s_T, Z),
\end{aligned}
\right.
\eeaa
where $ (\d W_{s_j})$ are independent copies of  $ \mathcal{N}( 0, \d s ) $. In the conventional gradient descent algorithm $\si$ is set to be $0$, wheras we add a (small) positive volatility in our model of the regularized optimal control problem. 
It is important to observe that the continuous-time version of the noised gradient descent algorithm follows exactly the dynamics of the system of mean-field Langevin equations \eqref{sys_Langevin}, where the horizon $s\in \dbR^+$ represents the iterations of gradient descents. 
 
We remark that the evaluation of the parameters $\Th^\cd_t$ on a given layer $t$ does not depend directly on the values of $\Th^\cd_{t'}$ on the other layers ($t'\neq t$), but only through the (empirical) law of $\Th^\cd_{t'}$. This `mean-field' dependence among the parameters is due to the average pooling on each layer in this particular structure, and is the starting point of our theoretical investigation.

Recall that we showed in Section \ref{subsec:contraction} that under a set of mild assumptions on the coefficients the marginal laws $(\nu^s)$ of  \eqref{sys_Langevin} converge to the optimal control of \eqref{reg_control}. It approximately implies that the stochastic gradient descent algorithm converges to the global minimizer.
 One of the main insights provided by this theory is the quantitative convergence rate. In particular, the theory ensures the exponential convergence once the coefficients satisfy $\frac{\sigma^2}{T}>\frac{2\g }{c\f(R_1)} $. Hopefully, it could shed some light on how to tune the coefficients in practice.
 
Further it remains crucial to justify that the output $\frac1T X^*_T$ given by the optimal parameter $\nu^*$ is a  good approximation to the label $Y=f(Z)$. In order for the contraction result to hold true we consider the horizon $T_\si := \frac{c\f(R_1)\si^2}{4\g}$.  Assume that 
$$X_0 = 0, \text{ $\phi$ does not depend on $t$ and } \ul c |\z| \le g(\z) \le \ol c |\z| +\e \si^2$$
for some small constant $\e>0$, and $\ol c \ge \ul c >0$. Then we have
\bea\label{eq:estimateexpress}
\E \Big| f(Z) - \frac{1}{T_\si} X^*_{T_\si}\Big|
\le  \frac{1}{T_\si}\E \Big| T_\si f(Z) -  X^*_{T_\si}\Big| 
\le  \frac{1}{\ul c T_\si} V^\si(\nu^*) =  \frac{1}{\ul c T_\si} \inf_{\nu\in \cV} V^\si(\nu).
\eea
Now consider the particular controls in the set $\cA := \big\{\nu\in \cV: ~\ell(A)=0,~\nu\mbox{-a.s.}  \big\}$ where $\ell(A)$ is the coefficient in front of $x$ in the activation function (see \eqref{eq:DNNcoeff}). The optimization over the set $\cA$ is equivalent to the optimization over $ \cV $ with the controlled process 
$d \tilde X_t = \int_{\dbR^m}\phi(0, a, Z) \nu_t(da)dt$.
Together with the observation in Remark \ref{rem:onelayer}, we obtain
\beaa
 \frac{1}{T_\si} \inf_{\nu\in \cV} V^\si(\nu)
 &\le&  \frac{1}{T_\si} \inf_{\nu\in \cA} V^\si(\nu)\\
  &\le& \frac{\e\si^2}{T_\si}  + \frac{1}{T_\si} \inf_{\nu\in \cA}\left( \ol c \, \E \Big|T_\si f(Z) - \tilde X_{T_\si}\Big|  +  \int_0^{T_\si} \int_{\dbR^m}\l |a|^2 \nu_t(da)dt +\frac{\si^2}{2}\int_0^{T_\si}\Ent(\nu_t)dt\right)\\
  &= &  \frac{4\g}{c\f(R_1)}\e  + \inf_{\nu_0 \in \cV} \left(  \ol c \, \E \Big| f(Z) - \int_{\dbR^m}\phi(0, a, Z) \nu_0(da)\Big|  + \int_{\dbR^m}\l |a|^2 \nu_0(da) +\frac{\si^2}{2}\Ent(\nu_0)\right)\\
  && \longrightarrow \frac{4\g}{c\f(R_1)}\e  + \ol c \inf_{\nu_0 \in \cV}  \E \Big| f(Z) - \int_{\dbR^m}\phi(0, a, Z) \nu_0(da)\Big|, 
  \q\q\mbox{as $\si,\l\rightarrow 0$. }
\eeaa
The last convergence is due to Proposition 2.3 of \cite{HRSS19}. Together with \eqref{eq:estimateexpress} we have
\beaa
\limsup_{\si,\l \rightarrow 0}\E \Big| f(Z) - \frac{1}{T_\si} X^*_{T_\si}\Big| 
\le \frac{4\g}{\ul c c\f(R_1)}\e  + \frac{\ol c }{\ul c } \inf_{\nu_0 \in \cV}  \E \Big| f(Z) - \int_{\dbR^m}\phi(0, a, Z) \nu_0(da)\Big|.
\eeaa
If one ignores the truncation function $\ell$ in \eqref{eq:DNNcoeff}, the universal representation theorem (see Theorem 1 of \cite{Hor91}) ensures that the value of the infimum on the right hand side is equal to $0$. Therefore we have shown that $\frac{1}{T_\si} X^*_{T_\si}$ is  an appropriate parametrized approximation for the label $f(Z)$.

\section{Wellposedness of the system of mean-field Langevin equations}\label{sec:proofWell}

\subsection{Wellposedness of the System}

\no {\bf Proof of Theorem \ref{thm:well}}:\q 
	 Let $ S>0 $. Given any process  $(\mu^s)_{s\in [0,S]}\in  C_p\big([0,S], \cV \big)$, we define for any $t \in [0,T]$ the process $ (\Th^s_t)_{s\in [0,S]} $ as the solution of the classical SDE:
	 \bea\label{eq:normalSDE}
\left\{
\begin{array}{ll}
	d \Th_t^s = -  \E\big[\nabla_a H(t, X_t^s(\mu), \Th_t^s, P_t^s(\mu), Z)\big] ds + \si dW_s, ~\q\q\q\q\mbox{for $s\in [0,S]$},\\
	\\
	\mbox{where}\q X_t^s(\mu) = X_0 + \int_0^t \int_{\dbR^m} \phi(r,X_r^s(\mu), a, Z) \mu_r^s (da) dr,\\
	\q\q\q~~ P_t^s(\mu) = \nabla_x G(X^s_T, Z) + \int_t^T \int_{\dbR^m} \nabla_x H\big(r, X_r^s(\mu), a, P_t^s(\mu), Z\big)  \mu_r^s(da)   dr, \q\mbox{for  $t\in [0,T]$,}
\end{array}
\right.
\eea
We are going to show that, for $ S $ small enough and independent of $ \Theta_t^0 $, the mapping $ (\mu^s)_{s\in [0,S]}\mapsto \big(dt\times\text{Law}(\Th^s_t)\big)_{s\in [0,S]}=:(\nu^s)_{s\in [0,S]} $ is a contraction on the space  $C_p\big([0,S], \cV \big)$ with the following metric:
	\begin{equation*}
		d^{T, S}_p(\nu, \mu) := \sup_{s\leq S}\ol\cW^T_p(\nu^s, \mu^s),
	\end{equation*}
and thus has a fixed point. Then the existence of a unique solution to the system \eqref{sys_Langevin} follows.

\ms
\no{\it Step 1.}\q  First we will show the following property for the image of the mapping: $(\nu^s)_{s\in [0,S]}\in C_p\big([0,S], \cV \big)$. It suffices to show that 
\bea\label{eq:L1W_cont}
\lim_{s'\rightarrow s} \int_0^T \dbE\big[ |\Th^{s'}_t-\Th^s_t |^p\big] dt = 0 \q\mbox{for all}\q s\in [0,S].
\eea
Since $\phi, \nabla_x H, \nabla_x G$ are all bounded, the processes $(X^s_t), (P^s_t)$ are both uniformly bounded. Further, by the assumption \eqref{assum:uniformLG}, the drift terms of the SDEs \eqref{eq:normalSDE} are of linear growth uniformly in $t$.  Then it follows from the standard estimate of SDE that
\beaa
 \dbE\big[ \sup_{s\in [0,S]}|\Th^s_t|^{p} \big] \le C \Big( \dbE \big[ |\Th^0_t|^{p} \big] +1\Big) \q \mbox{for all}\q t\in [0,T],
\eeaa
where $C$ is a constant independent of $t$.  
By the assumption \eqref{assum:Th0}, we have 
$\int_0^T\dbE\Big[ \sup_{s'\in [0,S]}  |\Th^{s'}_t-\Th^s_t |^{p}\Big] dt  <\infty$  for all $s\in [0,S]$.
Then \eqref{eq:L1W_cont} follows from the dominated convergence theorem.

\ms

\no {\it Step 2.}\q 	Let $(\mu^s)_{s\in [0,S]}, (\wt\mu^s)_{s\in [0,S]} \in  C_p\big([0,S], \cV \big)$  and denote $ (\Th^s_t)_{s\in [0,S]}  $ and $ (\wt\Th^s_t)_{s\in [0,S]}  $ the corresponding solutions of the SDEs (\ref{eq:normalSDE}), respectively. Denote $ \nu^s_t:= \text{Law}(\Th^s_t)$ and $ \wt\nu^s_t  := \text{Law}(\wt\Th^s_t)$. Denote $ K $ a common Lipschitz coefficient of $ \nabla_aH $, $ \nabla_xH $, $ \nabla_xG $ and $ \phi $. We have
\bea
		&& |\d X^s_t|:=|X^s_t(\mu) - X^s_t(\wt\mu)| \leq \int_{0}^{t} K\left(|X^s_r(\mu) - X^s_r(\wt\mu)| + \cW_1(\mu^s_r, \wt\mu^s_r)\right)\mathrm{d}r, \notag\\
	\mbox{and thus} \q &&|\d X^s_t| \leq Ke^{Kt} \ol\cW^t_1(\mu^s, \wt\mu^s) \le Ke^{Kt} \ol\cW^T_p(\mu^s, \wt\mu^s).\label{eq:estimateDx}
\eea
Similarly, we obtain	
\bea
 |\d P^s_t| &\leq& |\nabla_xG(X^s_T(\mu)) - \nabla_xG(X^s_T(\wt\mu))| + \int_{0}^{t} K \cW_1(\mu^s_r, \wt\mu^s_r) dr 
+ \int_0^t K\left( |\d X^s_r| + |\d P^s_r| \right)\mathrm{d}r \notag\\
& \leq& (K^2e^{KT}+K)\ol\cW^T_p(\mu^s, \wt\mu^s) + K^2\left(\int^t_0 e^{Kr}\mathrm{d}r\right)\ol\cW^T_p(\mu^s, \wt\mu^s)  + \int_{0}^{t}K|\d P^s_r|\mathrm{d}r \notag\\
&\leq & K(1+K)e^{KT}\ol\cW^T_p(\mu^s, \wt\mu^s) + \int_{0}^{t}K|\d P^s_r|\mathrm{d}r, \notag \\
\mbox{and thus} &&|\d P^s_t| \leq K(1+K)e^{2KT}\ol\cW^T_p(\mu^s, \wt\mu^s).\label{eq:estimateDp}
\eea
Then define $\d\Th^s_t:=\Th^s_t -\wt\Th^s_t$, and we can similarly estimate
	\beaa
		\int_{0}^{T}|\d\Th^s_t|^p \mathrm{d}t
		&\leq& C( p ) S^{p-1} K^p e^{C( p ) S^{p-1} K^p s} \int_{0}^{T}\int_{0}^{s} \Big(|\d X^r_t|^p + |\d P^r_t|^p \Big)\mathrm{d}r \mathrm{d}t \\
		&\leq& \tilde C( p ) S^{p-1} K^{p+1}( 1+K ) e^{C( p ) S^{p-1} K^p s +2KT} \int_{0}^{s}\ol\cW^T_p(\mu^r, \wt\mu^r)^p\mathrm{d}r,
	\eeaa
where $ C(p) $, $ \tilde C(p) $ are constants only depending on $p$.
By taking the expectation on both sides we get
	\beaa
		\ol\cW^T_p(\nu^s, \wt\nu^s)^p
		&\leq& \tilde C( p ) S^{p} K^{p+1}( 1+K ) e^{C( p ) S^{p} K^p +2KT} d^{T,S}_p(\mu,\wt\mu)^p \q\q \mbox{for}\q s\in[0,S].
	\eeaa
	Therefore, for $ S $ small enough, the mapping  $ (\mu^s)_{s\in [0,S]}\mapsto (\nu^s)_{s\in [0,S]} $ is a contraction.
	\qed

	\ms

Next we provide some useful  estimates for the solution to the system \eqref{sys_Langevin}. 

\begin{lem}\label{lem:stability_initial}
 Let Assumption \ref{assum:well} hold true.  Let $(\Th^0_t), (\tilde \Th^0_t)$ be two initial values satisfying \eqref{assum:Th0}, and denote by $(\nu^s_t), (\tilde \nu^s_t)$ the marginal laws of the solutions to the system \eqref{sys_Langevin}, respectively.  Then
$$ \ol \cW_p^T(\nu^s,\tilde \nu^s) \le C \ol\cW_p^T (\nu^0,\tilde \nu^0),$$
for some constant $C$ possibly depending on $s$. Moreover, if we further assume that  the functions $ \phi, \nabla_x H $ are globally Lipschitz continuous  in $ t $, then we have 
$$  \cW_p^T(\nu^s,\tilde \nu^s) \le C \cW_p^T (\nu^0,\tilde \nu^0).$$
\end{lem}
\begin{proof}
This first result is a direct result of an elementary estimate of SDE. As for the second one, it is enough to note that under the additional assumption, for each $s\in \dbR^+$ the mappings $(t,a)\mapsto \phi(t, X^s_t, a, Z )$ and $(t,a)\mapsto \nabla_x H(t,X^s_t, a, P^s_t,  Z )$ are both uniformly Lipschitz continuous, and thus
\beaa
|X^s_t -\tilde X^s_t | \le C \cW_1^T (\nu^s, \tilde\nu^s)\q\mbox{as well as}\q |P^s_t -\tilde P^s_t | \le C \cW_1^T (\nu^s, \tilde\nu^s).
\eeaa
The rest follows again from the standard estimate of SDE.
\qed
\end{proof}

\begin{lem}\label{lem:moment}
Let Assumptions \ref{assum:well} and \ref{assum:GF} hold true, and  $(\Th^0_t)_{t\in [0,T]}$ satisfy \eqref{assum:Th0p2}. 
Then we have
\bea
& \int_0^T\dbE\Big[  \sup_{s\in [0,S]}  |\Th^s_t|^p \Big] dt<\infty \q\mbox{for any $S\in \dbR^+$}, \label{eq:trivialmoment}\\
&\mbox{as well as} \q \int_0^T \sup_{s\in \dbR^+}  \dbE\big[ |\Th^s_t|^p \big] dt<\infty.\label{eq:uniform_moment}
\eea
\end{lem}

\begin{proof}
The result \eqref{eq:trivialmoment} follows from the standard SDE estimate, so its proof is omitted. By the It\^o formula, we have
\beaa
d |\Th^s_t |^p =  |\Th^s_t|^{p-2} \Big(-p \Th^s_t \cd \E \big[ \nabla_a H(t, X^s_t, \Th^s_t, P^s_t, Z)\big] +\frac{\si^2}{2} p(p+m-2) \Big) ds
			+ \si p  |\Th^s_t|^{p-2} \Th^s_t \cd dW_s.
\eeaa
Now recall the assumptions \eqref{assum:uniformLG} and \eqref{assum:dissipative} on $\nabla_a H$. We obtain:
\beaa
d |\Th^s_t |^p &\le&   p |\Th^s_t|^{p-2} \Big(C- \e |\Th^s_t|^2 1_{\{|\Th^s_t|\ge  M\}}  \Big) ds
			+ \si p   |\Th^s_t|^{p-2} \Th^s_t \cd dW_s \q\q\mbox{for some}\q M>0, \\
		&\le & p  |\Th^s_t|^{p-2} \Big((C +\e M^2)- \e |\Th^s_t|^2   \Big) ds
			+ \si p  |\Th^s_t|^{p-2} \Th^s_t \cd dW_s,
\eeaa
where $C$ does not depend on $t$. 
In the case $p=2$, it clearly leads to $\sup_{s\in\dbR^+}\dbE[|\Th^s_t|^2] \le C(1+\dbE[|\Th^0_t|^2])$, due to the Gronwall inequality. Then  \eqref{eq:uniform_moment} follows.  For general $p>2$, the result \eqref{eq:uniform_moment} is due to a simple induction. 
\qed
\end{proof}

\begin{lem}\label{lem:XLip}
Under the assumptions of Lemma \ref{lem:moment}, for each $t\in[0,T]$ the process $(X^s_t)_{s\in \dbR^+}$ in the system \eqref{sys_Langevin} is Lipschitz continuous in $s$.  
\end{lem}

\begin{proof}
Let $s'>s\ge 0$. Since the function $\phi$ is Lipschitz continuous in $(x,a)$, we have
\beaa
|\d X_t| : = |X^{s'}_t - X^s_t |
&=&
\Big| \int_0^t \Big( \dbE\big[\phi(r, X^{s'}_r, \Th^{s'}_r, Z)\big] - \dbE\big[\phi(r, X^s_r, \Th^s_r, Z)\big] \Big) dr \Big|\\
&\le &\int_0^t \Big(C |\d X_r  | + \Big| \dbE\big[\phi(r, X^s_r, \Th^{s'}_r, Z)-\phi(r, X^s_r, \Th^s_r, Z)\big] \Big|\Big)dr.
\eeaa
Further, by the It\^o formula, we have
\beaa
&&\Big| \dbE\big[\phi(r, X^s_r, \Th^{s'}_r, Z)-\phi(r, X^s_r, \Th^s_r, Z)\big]\Big|\\
&=&\Big| \dbE\Big[ \int_s^{s'} \Big(- \nabla_a \phi(r, X^s_r, \Th^u_r, Z) \cd \E\big[\nabla_a H(r, X^u_r, \Th^u_r, P^u_r, Z)\big] + \frac{\si^2}{2}\D_{aa}  \phi(r, X^s_r, \Th^u_r, Z)  \Big)du \Big]\Big|\\
&\le & C(s'-s)\Big(1+\sup_{u\in [s,s']} \dbE\big[|\Th^u_r|\big]\Big).
\eeaa 
The last inequality is due to the boundedness of $\nabla_a \phi, \D_{aa} \phi$ and the uniform linear growth of $\nabla_a H$ in $a$. 
Finally, it follows from Lemma \ref{lem:moment} and the Gronwall inequality that $|\d X_t| \le C(s'-s)$.
\qed
\end{proof}

Given a solution to the system of mean-field Langevin equations \eqref{sys_Langevin}, define
\bea\label{eq:bt}
b^t( s, a) := -\E\big[\nabla_a H(t, X^s_t, a, P^s_t, Z)\big].
\eea
It is easy to verify that under Assumptions \ref{assum:well} and \ref{assum:GF}, the function $b^t$ is continuous in $(s,a)$ and $C^3$ in $a$ for all $t\in [0,T]$. Due to a classical regularity result in the theory of linear PDEs (see e.g. \cite[p.14-15]{JKO98}), we obtain the following result.
\begin{lem}\label{lem:FP}
Let Assumptions \ref{assum:well} and \ref{assum:GF} hold true.  The marginal laws $(\nu^s_t)$ of the solution to \eqref{sys_Langevin}  are weakly continuous solutions  to the Fokker-Planck equations:
\bea\label{eq:FP}
\pa_s \nu = \nabla_a \cd (-b^t \nu  + \frac{\si^2}{2} \nabla_a \nu) \q\mbox{for}\q t\in [0,T].
\eea
In particular, we have that $(s,a)\mapsto \nu^s_t (a)$ belongs to $C^{1,2}\big((0,\infty)\times \dbR^m)\big)$.
\end{lem}
\ms

\subsection{Gradient Flow}

We start by providing an estimate of the value $\nabla_a \ln (\nu^s_t)$. First, the following result ensures that $\ln (\nu^s_t)$ is well defined.
 \begin{lem}\label{lem:finit_ent}
Let Assumption \ref{assum:well} and Assumption \ref{assum:GF} hold true and $ \dbE\big[ |\Th^0_t|^2 \big] < \infty $ for some $t\in[0,T]$. Denote by $ \dbQ^{\sigma}_t $ the scaled Wiener measure\footnote{Let $B$ be the canonical process of the Wiener space and $\dbQ$ be the Wiener measure, then the scaled Wiener measure $ \dbQ^{\sigma}: = \dbQ\circ (\si B)^{-1} $.} with initial distribution $ \nu^0_t $ and by $(\cF_s)_{s\in\dbR^+}$ the canonical filtration of the Wiener space. Then
 \begin{enumerate}
 	\item[i)] for any finite horizon $S>0$, the law of the solution to \eqref{sys_Langevin}, $ \nu_t := {\rm Law}\big((\Th^s_t)_{s\in\dbR^+}\big) $, is equivalent to $ \dbQ^{\sigma}_t $  on $ \cF_S $ and the relative entropy
 	\begin{equation}\label{eq:girsanov}
 		\int \ln \Big( \frac{\mathrm{d}\nu_t}{\mathrm{d}\dbQ^{\sigma}_t}\Big|_{\cF_S} \Big) d \nu_t 
		= \dbE\Big[ \int_0^S \big| b^t(s, \Th^s_t) \big|^2 ds \Big] 
		<+\infty.
 	\end{equation}
 	
 	\item[ii)] the marginal law $ \nu^s_t $ admits a density such that $ \nu^s_t>0 $ and $ \Ent(\nu^s_t)<+\infty $.
 \end{enumerate}
 \end{lem}
 
 \no The proof of these results is based on the Girsanov theorem and some simple moment estimates. It is similar to the proof of Lemma 6.1 in \cite{HRSS19} and thus omitted.  Further we have the following regularity result.
 
 \begin{lem}\label{lem:fisher_info}
 	For $t\in [0,T]$ and $(\nu_t^s)_{s\in \dbR^+}$ the marginal laws of the solution to \eqref{sys_Langevin}, under the same assumptions as in Lemma \ref{lem:finit_ent}, we have 		
 		\begin{equation}\label{eq:fisher_info}
 			\nabla_a \ln(\nu_t^s(a)) = -\frac1{s_0}\dbE\left[ \int_{0}^{s_0}\big(1-r\nabla_a b^t(r,\Th_t^{s-s_0+r})\big)\mathrm{d}W^{s-s_0}_r \Big| \Th_t^s=a \right] \q\mbox{for}\q s_0\in (0, s],
 		\end{equation}
 		where $ W^{s-s_0}_r := W_{s-s_0+r} - W_{s-s_0} $.
In particular,  for any $ s>0 $ we have 
\beaa
C:=\sup_{r\in [s,\infty)}\int_{\dbR^m } \big|\nabla_a\ln(\nu_t^r) \big|^2\nu_t^r(a)\mathrm{d}a<+\infty,
\eeaa
and $C$ only depends on the Lipschitz constant of $\nabla_a H$ with respect to $ a $. 
 \end{lem}
 \begin{proof}
The equality \eqref{eq:fisher_info} is shown in Lemma 6.2 in \cite{HRSS19}. The proof is based on Lemma 10.2 of the same paper and  \cite[Theorem 4.7 \& Remark 4.13]{Follmer}. Further, we have for all $s' \geq s$:
\beaa
\sup_{a\in \dbR^m}\big| \nabla_a \ln(\nu_t^{s'}(a)) \big|^2 
\le 
\inf_{s_0\in [s,s']}\frac{1}{s^2_0}\dbE\Big[ \int_{0}^{s_0}\big|1-r\nabla_a b^t(r,\Th_t^{s'-s_0+r})\big|^2\mathrm{d}r \Big].
\eeaa
Finally it is enough to note that $\nabla_a b^t$ is  bounded under the assumptions of the present Lemma.
 \qed
 \end{proof}

\begin{lem}\label{lem:integrability}
	Let Assumption \ref{assum:well} and Assumption \ref{assum:GF} hold true and $ \dbE\big[ |\Th^0_t|^2 \big] < \infty $. We have
\beaa
		&\int_{\dbR^m}|\nabla_a\nu^s_t(a)|\mathrm{d}a<+\infty, \q \int_{\dbR^m}|a\cdot\nabla_a\nu^s_t(a)|\mathrm{d}a<+\infty\q\mbox{for all}\q s>0,\\
		&\mbox{and}\q \int_s^{s'} \int_{\dbR^m}|\D_{aa}\nu^r_t(a)|\mathrm{d}a \mathrm{d}r <+\infty\q\mbox{for all}\q s'>s>0.
\eeaa
\end{lem}
\begin{proof}
	By the Young inequality, we have
	\beaa
	|\nabla_a\nu^s_t(a)| \leq \frac12 \nu^s_t(a) + \frac12 \left| \frac{\nabla_a\nu^s_t(a)}{\nu^s_t(a)} \right|^2\nu^s_t(a) &\text{and}& |a\cdot\nabla_a\nu^s_t(a)| \leq \frac12 |a|^2\nu^s_t(a) + \frac12 \left| \frac{\nabla_a\nu^s_t(a)}{\nu^s_t(a)} \right|^2\nu^s_t(a).
	\eeaa
	Since all the terms on the right hand sides are integrable, due to Lemma \ref{lem:fisher_info}, therefore so are $ \nabla_a\nu^s_t $ and $ a\cdot\nabla_a\nu^s_t $. Next, in order to prove the integrability of $ \D_{aa}\nu^s_t $, we apply It{\^o}'s formula:
	\begin{equation*}
		\mathrm{d}\ln(\nu^s_t(\Th^s_t)) = \left( \frac{\partial_s\nu^s_t(\Th^s_t)}{\nu^s_t(\Th^s_t)} + \frac{\nabla_a\nu^s_t(\Th^s_t)}{\nu^s_t(\Th^s_t)}\cdot b^t(s,\Th^s_t) + \frac{\sigma^2}{2}\D_{aa}(\ln(\nu^s_t(\Th^s_t)))  \right)\mathrm{d}s + \sigma\frac{\nabla_a\nu^s_t(\Th^s_t)}{\nu^s_t(\Th^s_t)}\mathrm{d}W_s.
	\end{equation*}
	Together with the Fokker-Planck equation \eqref{eq:FP}, we have
	\begin{equation}\label{eq:laplacian}
		\mathrm{d}\ln(\nu^s_t(\Th^s_t)) = \left( \sigma^2\frac{\D_{aa}\nu^s_t(\Th^s_t)}{\nu^s_t(\Th^s_t)} - \nabla_a\cdot b^t(s,\Th^s_t) - \frac{\si^2}{2} \frac{\left|\nabla_a\nu^s_t(\Th^s_t)\right|^2}{|\nu^s_t(\Th^s_t)|^2}  \right)\mathrm{d}s + \sigma\frac{\nabla_a\nu^s_t(\Th^s_t)}{\nu^s_t(\Th^s_t)}\mathrm{d}W_s.
	\end{equation}
	By Lemma \ref{lem:fisher_info}, we have 
$\dbE\left[ \int_{s}^{s'}\frac{\nabla_a\nu^r_t(\Th^r_t)}{\nu^r_t(\Th^r_t)}\mathrm{d}W_r \right] = 0$.
	Also recall that $ \nabla_a\cdot b^t(s,\Th^s_t) $ is bounded. Taking expectation on both sides of \eqref{eq:laplacian}, we obtain
	\beaa
		\si^2\int_s^{s'} \int_{\dbR^m}|\D_{aa}\nu^r_t(a)|\mathrm{d}a \mathrm{d}r
		&=& \dbE\left[\int_{s}^{s'}\sigma^2\frac{\D_{aa}\nu^r_t(\Th^r_t)}{\nu^s_t(\Th^s_t)}\mathrm{d}r\right]\\
		& \leq & \Ent(\nu^{s'}_t) - \Ent(\nu^s_t) + C\dbE\left[ \int_{s}^{s'}\left(1 + \frac{\left|\nabla_a\nu^r_t(\Th^r_t)\right|^2}{\nu^r_t(\Th^r_t)}  \right)\mathrm{d}s\right].
	\eeaa
	By Lemma \ref{lem:finit_ent} and \ref{lem:fisher_info}, the right hand side is finite.
	\qed
\end{proof}

\no Based on the previous integrability results, the next lemma follows from the integration by parts.

\begin{lem}\label{lem:IPP}
	Under the assumptions of Lemma \ref{lem:integrability} we have for $s>0$
\beaa
		&\int_{\dbR^m} \D_{aa} H(t,X^s_t,a,P^s_t,Z)\nu^s_t(a)\mathrm{d}a = -\int_{\dbR^m} \nabla_aH(t,X^s_t,a,P^s_t,Z) \cd\nabla_a\nu^s_t(a)\mathrm{d}a\q\mbox{for all}\q s>0,\\
		&\int_s^{s'}\int_{\dbR^m}\D_{aa} \big(\ln \nu^r_t(a) \big)\nu^r_t(a)\mathrm{d}a \mathrm{d} r = -\int_s^{s'} \int_{\dbR^m}\left|  \nabla_a \ln \nu^r_t(a) \right|^2\nu^r_t(a)\mathrm{d}a\mathrm{d} r \q\mbox{for all}\q s'>s>0.
\eeaa
\end{lem}

\no{\bf Proof of Theorem \ref{thm:gradient_flow}}.\q It follows from  Lemma \ref{lem:XLip} that there exists a bounded process $ (U^s_t) $ such that 	$\mathrm{d}X^s_t = U^s_t\mathrm{d}s$. On the other hand, note that 
$X^s_t = X_0 + \int_0^t \dbE \big[ \phi(r, X^s_r, \Th^s_r, Z)\big]dr$.
By the It{\^o} formula, we get
\beaa
	\frac{\mathrm{d}X^s_t}{\mathrm{d}s} =
	 \int^t_0\int_{\dbR^m} \Big( \nabla_a\phi(r,X^s_r,a,Z) b^r(s,X^s_r) 
	 + \nabla_x\phi(r,X^s_r,a,Z) U^s_r 
	 + \frac{\sigma^2}{2}\D_{aa}\phi(r,X^s_r,a,Z) \Big) \nu^s_r (\mathrm{d}a)\mathrm{d}r,
\eeaa
where $b^r$, defined as in \eqref{eq:bt}, is the drift term of the diffusion $(\Th^s_r)_{s\in \dbR^+}$. In particular, we have
\bea\label{eq:derv_U}
	\frac{\mathrm{d}U^s_t}{\mathrm{d}t} = \int_{\dbR^m }\Big( \nabla_a\phi(t,X^s_t,a,Z) b^t(s,X^s_t) + \nabla_x\phi(t,X^s_t,a,Z) U^s_t + \frac{\sigma^2}{2}\D_{aa}\phi(t,X^s_t,a,Z)  \Big) \nu^s_t (\mathrm{d} a).
\eea
Then note that
$V(\nu^s) = \E\left[\int_{0}^{T}\dbE\big[ L(t ,X^s_t, \Th^s_t, Z)\big] \mathrm{d}t + G(X^s_T, Z)\right]$. 
Again by the It{\^o} formula, we have
\begin{multline}\label{eq:derv_V}
\frac{\mathrm{d}V(\nu^s)}{\mathrm{d}s} = \E \Big[\int_{0}^{T}\int_{\dbR^m}\Big( \nabla_a L(t,X^s_t,a,Z)\cd b^t(s,a) + \frac{\si^2}{2} \D_{aa}L(t,X^s_t,a,Z)   \\
+ \nabla_xL(t,X^s_t,a,Z) \cd U^s_t\Big)  \nu^s_t(\mathrm{d}a)\mathrm{d}t + \nabla_xG(X^s_T, Z) \cd U^s_T \Big].
\end{multline}
Recall \eqref{eq:derv_U} and the dynamic of $(P^s_t)_{t\in [0,T]}$ in \eqref{sys_Langevin}. By integration by parts, we have
\beaa
\nabla_xG(X^s_T, Z)\cd U^s_T
&=& \int_{0}^{T}\int_{\dbR^m} \Big(-U^s_t \cd \nabla_xH(t,X^s_t,a,P^s_t,Z) + P^s_t \cd \nabla_x\phi(t,X^s_t,a,Z)U^s_t \\
&\q& \q\q+   P^s_t \cd \nabla_a\phi(t,X^s_t, a,Z)b^t(s,a)+ P^s_t \cd \frac{\sigma^2}{2} \D_{aa}\phi(t,X^s_t,a,Z) \Big)\nu^s_t(\mathrm{d}a)\mathrm{d}t\\
& =&\int_{0}^{T}\int_{\dbR^m}\Big(- U^s_t \cd \nabla_xL(t,X^s_t,a,Z)  \\
&\q& \q\q +   P^s_t \cd \nabla_a\phi(t,X^s_t, a,Z)b^t(s,a) +P^s_t \cd \frac{\sigma^2}{2} \D_{aa}\phi(t,X^s_t,a,Z) \Big)\nu^s_t(\mathrm{d}a)\mathrm{d}t .
\eeaa
Together with \eqref{eq:derv_V}, we obtain
\beaa
\frac{\mathrm{d}V(\nu^s)}{\mathrm{d}s}
&=&\E\Big[ \int_0^T \int_{\dbR^m} \Big( b^t(s, a) \cd \nabla_a H(t,X^s_t,a,P^s_t,Z ) + \frac{\si^2}{2} \D_{aa} H(t,X^s_t,a,P^s_t,Z ) \Big)  \nu^s_t(\mathrm{d}a)\mathrm{d}t \Big]\\
&=&  \int_0^T \int_{\dbR^m} \Big( - \big|b^t(s, a)\big|^2 + \E\Big[ \frac{\si^2}{2} \D_{aa} H(t,X^s_t,a,P^s_t,Z ) \Big]\Big)  \nu^s_t(\mathrm{d}a)\mathrm{d}t.
\eeaa
Further by Lemma \ref{lem:IPP}, we have for $s>0$
\bea\label{eq:v}
\frac{\mathrm{d}V(\nu^s)}{ds}
= \int_{0}^{T}\int_{\dbR^m}\left( - \big|b^t(s,a)\big|^2 + \frac{\si^2}{2}b^t(s,a)\cd \nabla_a\ln \nu^s_t(a) \right)\nu^s_t(\mathrm{d}a)\mathrm{d}t .
\eea

On the other hand, recall formula \eqref{eq:laplacian}. By taking expectation on both sides and applying Lemma \ref{lem:IPP}, we obtain for any $ s>0 $:
\bea\label{eq:entropy}
\frac{\mathrm{d}\Big(\frac{\si^2}{2}\int^T_0\Ent(\nu^s_t)\mathrm{d}t \Big)}{\mathrm{d}s}
= \int^T_0\int_{\dbR^m}\left( \frac{\si^2}{2}\nabla_a\ln \nu^s_t(a)  \cd b^t(s,a)  - \frac{\si^4}{4}\big| \nabla_a\ln \nu^s_t(a)  \big|^2 \right)\nu^s_t(\mathrm{d}a)\mathrm{d}t.
\eea
Summing up \eqref{eq:v} and \eqref{eq:entropy}, we finally obtain \eqref{eq:gradient_flow}.
\qed

\section{Proof for the Convex Case}\label{sec:convex}

\subsection{Sufficient First Order Condition}

We are going to apply a standard variational calculus argument in order to derive the sufficient condition for being the optimal control of \eqref{reg_control}. 
\ms

\no{\bf Proof of Theorem \ref{thm:FOC}}.\q 
Take a $ \nu \in \cV$ such that $ \nu $ is absolutely continuous with respect to Lebesgue measure (otherwise $ \int_0^T\Ent(\nu_t)dt =+\infty $), and thus absolutely continuous with respect to the measure $\nu^*$. Denote $ X^*$ and $ X $ the controlled  processes with $ \nu^* $ and $ \nu $, respectively, and define $\d X: = X- X^*$ and $\d \nu:= \nu-\nu^*$.
By the assumption on convexity of the coefficients, we have
	\bea
		\d V & := & \E\left[\int_{0}^{T}\int_{\dbR^m}\Big( L(t,X_t,a, Z)\nu_t( d a)- L(t,X^*_t,a, Z) \nu^*_t( d a)  \Big)d t + G(X_T, Z)-G(X^*_T, Z)\right] \notag \\
		& \geq & \E\Big[\int_{0}^{T}\Big( \nabla_x\ell(t,X^*_t, Z)\cd \d X_t +  \int_{\dbR^m} L(t,X^*_t, a, Z) \d\nu_t( d a)\Big) d t + \nabla_x G(X^*_T, Z)\cd \d X_T \Big]. \label{neq_sufficient_condition}
	\eea
Recall the adjoint process $P^*$ defined in \eqref{eq:adjointproc}. 
	By integration by parts, we have
	\beaa
		\nabla_x G(X^*_T) \cd \delta X_T = P^*_T \cd \delta X_T  =\int_{0}^{T}\left(\int_{\dbR^m}P^*_s\cd \phi(s, a)\delta\nu_s( d a) - \nabla_x \ell(s,X^*_s,Z)\cd \delta X_s \right) d s.
	\eeaa
Together with \eqref{neq_sufficient_condition}, it leads to
	\beaa
	\d V &\geq&  \E\left[\int_{0}^{T}\int_{\dbR^m} \Big(P^*_t \cd \phi(t,a) + L(t,X^*_t, a,Z)\Big)\d\nu_t(da) d t\right].
	\eeaa
Further, we are going to compute the difference of the relative entropies. Since $\nu$ is absolutely continuous with respect to $\nu^*$, we may define the Radon-Nikodym derivative $f_t:=\frac{\nu_t}{\nu^*_t}$.
	Denote $h(x) = x\ln(x)$ and note that $h(x) \ge x-1$ for all $x\in \dbR^+$. 
	We have
\begin{multline*}
\Ent(\nu_t) - \Ent(\nu_t^\ast)
 = \int_{\mathbb R^d} \left(\nu_t \ln \nu_t - \nu_t^\ast \ln \nu_t^\ast \right)\, d x
	 =\int_{\mathbb R^d} (\nu_t - \nu_t^\ast)\ln\nu_t^\ast\, d x + \int_{\mathbb R^d} \nu_t \left(\ln \nu_t - \ln \nu_t^\ast \right)\, d x\\
= \int_{\mathbb R^d} (f_t-1)\nu_t^\ast \ln \nu_t^\ast\, d x +\int_{\mathbb R^d} h(f_t)\nu_t^\ast \, d x
	 \geq \int_{\mathbb R^d} (f_t-1)\nu_t^\ast \ln \nu_t^\ast\, d x + \int_{\mathbb R^d} (f_t-1)\nu_t^\ast \, d x = \int_{\dbR^m}\ln(\nu^*_t(a))\d \nu_t( d a).
\end{multline*}
The last equality is due to $\int_{\mathbb R^d} (f_t-1)\nu_t^\ast \,  d x = \int_{\mathbb R^d} (\nu_t-\nu_t^\ast) \,  d x = 0$. Finally, by \eqref{first_order_condition} we have
\beaa
		V^\sigma(\nu) - V^\sigma(\nu^*) &\geq& \E\left[\int_{0}^{T}\int_{\dbR^m} \Big(P^*_t\cd \phi(t,a) + L(t,X^*_t, a, Z) + \frac{\sigma^2}{2}\ln(\nu^*_t(a))\Big)\d\nu_t (da) d t\right] \\
		& = &\E\left[\int_{0}^{T}\int_{\dbR^m} \Big(H(t, X^*_t, a, P^*_t, Z) + \frac{\sigma^2}{2}\ln(\nu^*_t(a))\Big)\d\nu_t (da) d t\right] = 0.
\eeaa
\qed

\subsection{Convergence Towards the Invariant Measure}
In order to prove that there exists an invariant measure of \eqref{sys_Langevin} equal to the minimizer of $ V^\si $, we  follow the same strategy as in \cite{HRSS19}. For readers' convenience, we shall provide a brief proof. The main ingredients of the proof are LaSalle's invariance principle (see e.g. \cite[Theorem 4.3.3]{Henry81}) and the  HWI inequality (see \cite[Theorem 3]{OV00}).  Let $ (\nu^s)_{s\in\dbR^+} $ be the flow of marginal laws of the solution of \eqref{sys_Langevin}, given an initial law $ \nu^0 $. Define a dynamic system $ \cS(s)\left[\nu^0\right] := \nu^s $. We shall consider the following $ \omega $-limit set:
\begin{equation*}
	\omega(\nu^0) := \left\{ \nu\in\cV:\q \text{there exists }s_n\rightarrow+\infty\text{ such that } \cW_2^T\left(\cS(s_n)\left[\nu^0\right], \nu\right)\rightarrow 0  \right\}.
\end{equation*}

\begin{prop}[Invariance Principle]
	Assume that Assumption \ref{assum:well} and Assumption \ref{assum:GF} hold true and $ \nu^0 $ satisfies \eqref{assum:Th0p2} for some $p>2$. Then the set $ \omega(\nu^0) $ is non-empty, compact and invariant, that is
	\begin{enumerate}
		\item[i)] for any $ \nu\in\omega(\nu^0) $, we have $ S(s)\left[\nu\right]\in\omega(\nu^0) $ for all $ s\in\dbR^+ $;
		
		\item[ii)] for any $ \nu\in\omega(\nu^0) $ and all $ s\in\dbR^+ $, there exists $ \nu'\in\omega(\nu^0) $ such that $ S(s)\left[\nu'\right]=\nu $.
	\end{enumerate}
\end{prop}

\begin{proof}
It is important to note that
\begin{itemize}
\item the mapping $\nu^0 \mapsto  \cS(s)\left[\nu^0\right]$ is $\cW_2^T$-continuous, due to Lemma \ref{lem:stability_initial};
\item the mapping $s\mapsto  \cS(s)\left[\nu^0\right]$ belongs to $C_2\big(\dbR^+, \cV\big)$, due to Theorem \ref{thm:well};
\item the set $\big\{ \cS(s)[\nu^0],~s\in \dbR^+\big\}$ belongs to a $\cW_2^T$-compact set, due to Lemma \ref{lem:moment}.
\end{itemize}
The rest follows the standard argument for LaSalle's invariance principle (see e.g. \cite[Theorem 4.3.3]{Henry81} or \cite[Proposition 6.5]{HRSS19}).
\qed
\end{proof}

\ms

\no{\bf Proof of Theorem \ref{thm:convex_convergence}}.\q
As in the Step 1  of the proof to \cite[Theorem 2.10]{HRSS19},
 using the invariance principle we can prove the existence of a convergent subsequence of the measure flow $(\nu^{s_n})_{n\in\dbN}$ such that
$\cW_2^T\big(\nu^{s_n}, \nu^* \big) \rightarrow 0$, where $\nu^* = \arg\min_\nu V^\si(\nu)$ and satisfies
\beaa
\nu^*(t, a) = C \exp\Big(-\frac{2}{\si^2}H(t, X^*_t, a, P^*_t, Z) \Big).
\eeaa
In particular, $\nu^*$ is log-semiconcave, because one may easily verify that the gradient of the mapping $(t,a)\mapsto H(t, X^*_t, a, P^*_t, Z)$ is Lipschitz continuous. By the HWI inequality we have
\beaa
\int  \big(\ln \nu^{s_n} -\ln  \nu^*\big) \nu^{s_n}(dt, da)  
\le 
T^{-\frac12}\cW_2^T \big( \nu^{s_n} , \nu^* \big) \left(\sqrt{\cI_n/T} + C T^{-\frac12}\cW_2^T \big( \nu^{s_n} , \nu^* \big) \right),
\eeaa
where $\cI_n$ is the relative Fisher information defined as
\beaa
\cI_n &: =& \int \Big| \nabla_a \ln \nu^{s_n} - \nabla_a \ln \nu^*   \Big|^2 \nu^{s_n}(dt, da) \\
&=&   \int \Big| \nabla_a \ln \nu^{s_n} +\frac{2}{\si^2}\nabla_a H(t, X^*_t, a, P^*_t, Z)  \Big|^2 \nu^{s_n}(dt, da)\\
&\le & 2 \int \big| \nabla_a \ln \nu^{s_n} \big|^2 \nu^{s_n}(dt, da) + C \Big(1+\int |a|^2 \nu^{s_n}(dt, da)\Big).
\eeaa
Then it follows from Lemma \ref{lem:fisher_info} and \ref{lem:moment} that $\sup_n \cI_n <\infty$.  Together with the fact that $\cW_2^T\big(\nu^{s_n}, \nu^* \big) \rightarrow 0$,  we have
\beaa
 \limsup_{n\rightarrow\infty} \Ent(\nu^{s_n}) - \Ent(\nu^*) = \limsup_{n\rightarrow\infty} \int  \big(\ln \nu^{s_n} -\ln  \nu^*\big) \nu^{s_n}(dt, da)  \le 0.
\eeaa
Since $\Ent$ is $\cW_2^T$-lower-semicontinuous, we have $ \lim\limits_{n\rightarrow\infty} \Ent(\nu^{s_n}) = \Ent(\nu^*) $, and thus $\lim\limits_{s\rightarrow\infty} V^\si(\nu^s) = V^\si (\nu^*)$. We have proved that $V^\si$ is a continuous Lyapunov function along the trajectory of $(\nu^s)$. Further we can conclude the proof using the standard argument (see e.g. Step 3 of the proof of Theorem 2.10 \cite{HRSS19}).
\qed

\section{Proofs for the Contraction Case}\label{sec:contraction}

We first provide a sufficient condition for the regularized control problem \eqref{reg_control} to have at least one optimal control.

\begin{lem}\label{lem:exist_optimal}
Assume that there exists $\nu\in \cV$ such that $V^\si(\nu)<\infty$, and that there is  a function $U:\dbR^m\rightarrow\dbR^+$ such that $ \int_{\dbR^m}e^{-U(a)}da <\infty$ and $\bar V(\nu) := V(\nu) - \frac{\si^2}{2}\int_0^T \int_{\dbR^m} U(a) \nu_t(da)dt$ is bounded from below and weakly lower-semicontinuous. 
Then $\argmin\limits_{\nu} V^\si(\nu) \neq \emptyset$. 
\end{lem}

\begin{proof}
Let $\bar\nu\in \cV$ such that $V^\si(\bar\nu)<\infty$. Denote $C_0 := \inf_{\nu} \bar V(\nu)$ and $\bar C:=  V^\si(\bar\nu) - C_0$. Recall that
\beaa
V^\si &=& \bar V + \frac{\si^2}{2} \int_0^T \Big( \Ent(\nu_t)  + \int_{\dbR^m} U(a) \nu_t(da)\Big)dt ~= ~ \bar V + \frac{\si^2}{2} I (\nu),\\
\mbox{where}&& I (\nu) :=\int_0^T  \int_{\dbR^m}\ln\left(\frac{\nu_t(a)}{e^{-U(a)}}\right) \nu_t(da) dt.
\eeaa
Note that $I(\nu)$ is the relative entropy of $\nu$ with respect to the measure $dt\times e^{-U(a)}da$. Therefore, $I$ is weakly lower-semicontinuous (so is $V^\si$) and 
 the sublevel set  $\cK : = \big\{\nu \in \cV : ~ \frac{\sigma^2}{2} I(\nu)\le \bar C \big\}$ is weakly compact (see e.g. \cite[Lemma 1.4.3]{DE97}).
%
Further note that $\big\{ \nu\in \cV: ~ V^\si(\nu) \le V^\si(\bar\nu)\big\} \subset \cK$, so $
\inf_{\nu\in \cV} V^\si(\nu) = \inf_{\nu\in \cK} V^\si(\nu)$.  
Since $V^\si$ is weakly lower-semicontinuous and $\cK$ is weakly compact, there exists a global minimizer in $\cK$.
\qed
\end{proof}
\ms

Next we prove the necessary condition of being an optimal control. 
\ms

\no{\bf Proof of Proposition \ref{prop:NC}}. \q 
Since $V^\si(\nu^*)<\infty$,  we know that $\int_0^T \Ent(\nu^*_t)dt<\infty $. In particular, $ \nu^* $ is absolutely continuous with respect to the Lebesgue measure. 

\ms
\no {\it Step 1}.\q Let  $ \nu\in\cV $ be a measure such that $\int_0^T \Ent(\nu_t) dt <\infty$, in particular, it is absolutely continuous with respect to the Lebesgue measure. Define $\nu^\eps := (1-\eps)\nu^* + \eps \nu  $ for $\eps>0$. By standard variational calculus we have

\beaa
\lim\limits_{\eps\rightarrow 0}\frac{V(\nu^\eps)-V(\nu^*)}{\eps} &=&  \int_{0}^{T}\int_{\dbR^m} \E \big[ H(t, X^*_t, a, P^*_t, Z) \big]\big(\nu(da)-\nu^*_t(da)\big) d t.
\eeaa
Further, define the function $h(x): =x \ln x$. We have
\beaa
\frac1\eps\int_0^T\left(\Ent(\nu^\eps_t)-\Ent(\nu^*_t)\right) d t
&=& \frac{1}{\eps}\int_{0}^{T}\int_{\mathbb{R}^{m}}\Big(h\left( \nu^\eps_t(a) \right) - h\left( \nu^*_t(a) \right)\Big) d a d t \\
\eeaa
Since the function $h$ is convex, we note that 
\beaa
\frac{1}{\eps} \big(h\left( \nu^\eps_t(a) \right) - h\left( \nu^*_t(a) \right)\big) \le h(\nu_t(a)) -h(\nu^*_t(a))\q\mbox{for all}\q \e\in(0,1).
\eeaa
The right hand side of the inequality above is integrable because both $\int_0^T \Ent(\nu_t)dt $ and $ \int_0^T \Ent(\nu^*_t)dt $ are finite. Therefore, by Fatou's lemma we obtain
\bea\label{eq:nece_cal}
0 \leq \limsup_{\eps\rightarrow 0}\frac{V^\si(\nu^\eps)-V^\si(\nu^*)}{\eps}
 \le \int^T_0\int_{\mathbb{R}^{m}}\left( \E \big[ H(t, X^*_t, a, P^*_t, Z) \big] +\frac{\si^2}{2}  \ln\big(\nu^*_t(a)\big) \right)\big(\nu_t( d a)-\nu^*_t( d a)\big) d t.
\eea
\ms

\no{\it Step 2}.\q We are going to show that for Leb-a.s. $t$ 
\bea\label{eq:nece_constant}
\Xi_t(a):= \E \big[ H(t, X^*_t, a , P^*_t, Z) \big] +\frac{\si^2}{2}  \ln\big(\nu^*_t(a)\big)\q\mbox{is equal to a constant $\nu^*_t$-a.s.} 
\eea
Define the mean value $\ol c_t : = \int_{\dbR^m} \Xi_t(a) \nu^*_t(da)$ and let $\e, \e'>0$. Consider the measure $\nu\in \cV$ absolutely continuous with respect to $\nu^*$ such that $\nu_t = \nu^*_t$ if $\nu^*_t[\Xi_t\le \ol c_t -\e] <\e'$, otherwise
\beaa
\frac{d\nu_t}{d\nu^*_t} = \frac{1_{\Xi_t \le \ol c_t -\e}}{\nu^*_t[\Xi_t \le \ol c_t -\e] }.
\eeaa
Note that $\Xi_t \le \ol c_t-\e$, $\nu_t$-a.s. for $t$ such that $\nu^*_t[\Xi_t\le \ol c_t -\e] \ge \e'$. Then we have
\beaa
\int_0^T\int_{\dbR^m} \Xi_t(a) \big(\nu_t(da) -\nu^*_t(da)\big)dt \le -\e \int_0^T 1_{\nu^*_t[\Xi_t\le \ol c_t -\e] \ge \e'}dt.
\eeaa
Together with \eqref{eq:nece_cal}, we conclude
$\nu^*_t[\Xi_t\le \ol c_t -\e] <\e'$ for Leb-a.s. $t\in [0,T]$.
Since this holds true for arbitrary $\e', \e>0$, we obtain \eqref{eq:nece_constant}.

\ms
\no{\it Step 3}.\q We are going to show that $\nu^*$ is equivalent to the Lebesgue measure. First we provide an estimate for the constant $\ol c_t$ above. Since $\nu^*_t$ is a probability measure, we have
\bea\label{eq:vstarprob}
\int_{\dbR^m}\exp\left(\frac{2\Big(\ol c_t -\E\big[ H(t, X^*_t, a , P^*_t, Z) \big]  \Big)}{\si^2}\right) da =1.
\eea
 Moreover, since $(t,z)\mapsto H(t,0,0,0,z)$ is bounded and $a\mapsto \nabla_a H(t,x,a,p,z)$ is uniformly Lipschitz continuous, we have 
\bea\label{eq:easyboundH}
\sup_{t,z} \big| H(t,X^*_t,a,P^*_t,z) \big| \le C(1 + |a|^2).
\eea
On the other hand, following the dissipative assumption \eqref{assum:dissipative}, we may easily prove that there are constants $C,C'>0$ such that for all $(t,a)$
\bea\label{eq:dissboundH}
H(t,X^*_t,a,P^*_t,z) \ge -C + C'|a|^2.
\eea
Together with \eqref{eq:vstarprob} and \eqref{eq:easyboundH}, we prove that $(\ol c_t)_{t\in[0,T]}$ is bounded. 

Now suppose that $\nu^*$ is not equivalent to the Lebesgue measure. Then there is a set $\cK\in [0,T]\times \dbR^m$ such that $\nu^*(\cK)=0$ (so $\ln\nu^* = -\infty$ on $\cK$) and $\mbox{Leb}[\cK]>0$. It follows from \eqref{eq:nece_cal} that 
$0\le C - \int_\cK \infty d\nu$.
Since we may choose $\nu$ having positive mass on $\cK$, it is a contradiction. Therefore $\nu^*$ must be equivalent to the Lebesgue measure. 

\ms

\no{\it Step 4}.\q Since  $\nu^*$ is equivalent to the Lebesgue measure, together with \eqref{eq:nece_constant} we obtain \eqref{eq:NC}, and thus $\nu^*_t$ is a stationary solution to the Fokker-Planck equation \eqref{eq:FP} for Leb-a.s. $t\in [0,T]$.
Therefore $\nu^*$ is an invariant measure of the mean-field Langevin system \eqref{sys_Langevin}. 
\qed
\ms

In order to prove the main Theorem \ref{thm:contraction}, the main ingredient is the reflection coupling  in Eberle \cite{eberle11}. For the mean-field system, we shall adopt the reflection-synchronous coupling similar to \cite{eberle2019quantitative}. 

We fix a parameter $\e>0$. Introduce the Lipschitz functions ${\rm rc}: \dbR^m \times \dbR^m\rightarrow [0,1]$ and ${\rm sc}: \dbR^m \times \dbR^m\rightarrow [0,1]$ satisfying
\beaa
{\rm sc}^2 (x, y) + {\rm rc}^2(x,y) =1, \q
{\rm rc}(x,y) =1 ~~\mbox{for $|x-y|\ge \e$}, \q 
{\rm rc}(x,y) =0 ~~\mbox{for $|x-y|\le \e/2$}.
\eeaa
 Let $(\nu^0_t)_{t\in [0,T]}$ and $(\tilde \nu^0_t)_{t\in [0,T]}$ be two initial measures, and $(W^1_s), (W^2_s)$ be two independent Brownian motions. For the given  $(\nu^0_t), (\tilde\nu^0_t)$ we construct the drift coefficients $(b^t), (\tilde b^t)$ as in \eqref{eq:bt}, respectively. Further, for a fixed\footnote{We are not defining the coupling for the system of SDE's, but for a single SDE with the fixed label $t$.} $t\in [0,T]$,  define the coupling $\Si_t=(\Th_t, \tilde\Th_t)$ as the solution to the standard SDE
 \beaa
 && d\Th^s_t = b^t(s, \Th^s_t) ds + \rc (\Si^s_t) \si dW^1_s + \syc(\Si^s_t)\sigma dW^2_s,\\
 && d\tilde\Th^s_t = \tilde b^t(s, \tilde\Th^s_t) ds + \rc (\Si^s_t) \big({\rm Id} - 2 e_s \langle e_s, \cd\rangle \big)\sigma dW^1_s + \syc(\Si^s_t)\sigma dW^2_s,
 \eeaa
where $e_s:= \frac{\Th^s_t -\tilde\Th^s_t}{|\Th^s_t -\tilde\Th^s_t| }$ for $\Th^s_t \neq \tilde\Th^s_t$, otherwise $e_s:=\hat e$ some arbitrary fixed unit vector in $\dbR^m$. Next, we construct a concave increasing function $f$ as in the proof of \cite[Theorem 2.3]{eberle2019quantitative}. 
Let
\begin{equation*}
	f(r) := \int_{0}^{r}\f(s)g(s\wedge R_2) d s,\q
	\mbox{where}\q g(r):= 1-\frac{c}{2}\int_0^r \Phi(s)\f(s)^{-1}ds,
\end{equation*}
and the function $ \f  $ and  the constant $R_2$ are defined as in the statement of Theorem \ref{thm:contraction}.
Note that by definition $ \kappa^+(r)=0 $ for any $ r\geq R_1 $, so $ (\f(r))_{r\ge R_1} $ is a constant  and the function $f$ is linear on $[R_2,\infty)$. Furthermore, $ f $ is twice continuously differentiable on $(0,R_2)\cup(R_2,+\infty)$. Clearly for $ r\in(0,R_2) $ we have 
\bea
& 2\sigma^2 f''(r) = -r \k^+(r) f'(r) - c\sigma^2 \Phi(r) \le -r \k(r) f'(r) - c\sigma^2 f(r) \label{eq:fODE}
\eea
and for $ r\in\dbR^+ $
\bea
& r\f(R_1) \le \Phi (r) \le 2 f(r) \le 2\Phi (r) \le 2r. \label{eq:bound_f}
\eea
Next we prove the following inequality similar to \eqref{eq:fODE} on $(R_2, \infty) $ where $ f''=0 $:
$0 \le -r \k(r) f'(r) - c\sigma^2 f(r)$.
Recall  that $ (\f(r))_{r\ge R_1} $ is a constant and thus we have $ \Phi(r) = \Phi(R_1) + \f(R_1)(r-R_1) $. Since $ \Phi(R_1) \ge \f(R_1)R_1$, we have
\bea\label{eq:estimatePhi}
\frac{\Phi(r)}{\Phi(R_2)}  = \frac{ \Phi(R_1) -\f(R_1)R_1 + \f(R_1)r}{\Phi(R_1) -\f(R_1)R_1 + \f(R_1)R_2}\le \frac{r}{R_2},
\q\mbox{for $r\ge R_2$}.
\eea
 Furthermore, it is easy to verify that
\begin{align}\label{eq:estimatec-1}
	c^{-1}\ge \int_{R_1}^{R_2}\Phi(s)\f^{-1}(s) d s 
	&\ge \frac{\f^{-1}(R_1) \Phi(R_2)(R_2-R_1)}{2}.
\end{align}
Also note that $g(R_2) = \frac12$ due to the definition of $c$, and thus  $ f'(r) = \frac{\f(R_1)}{2} $ for $ r\geq R_2 $. Together with \eqref{eq:estimatePhi}, \eqref{eq:estimatec-1} and the definition of $ R_2 $, we have
\begin{equation*}
	r\k(r)f'(r) \leq -2\si^2\frac{r\f(R_1)}{(R_2-R_1)R_2}
	\leq -2\si^2\frac{\Phi(r)\f(R_1)}{(R_2-R_1)\Phi(R_2)}\leq -c\si^2\Phi(r) \leq -c\si^2f(r). 
\end{equation*}

\ms
\no {\bf Proof of Theorem \ref{thm:contraction}}.\q {\it Step 1.}\q We first use an argument similar to that of the proof of Theorem 2.3 in \cite{eberle2019quantitative} to obtain some estimates concerning the coupling. As usual in the contraction result, we choose the coupling $ ( \Th^0_t, \tilde \Th^0_t ) $ so that
\bea\label{eq:fWlower}
\cW_1 (\nu^0_t, \tilde\nu^0_t) = \dbE \big[|\Th^0_t - \tilde \Th^0_t|\big] 
\ge \dbE \Big[f\big(|\Th^0_t - \tilde \Th^0_t|\big)\Big]  .
\eea
The last inequality is due to \eqref{eq:bound_f}. On the other hand, for all $s\ge 0$ we have
\bea\label{eq:fWupper}
\cW_1 (\nu^s_t, \tilde\nu^s_t) \le  \dbE \big[|\Th^s_t - \tilde \Th^s_t|\big]
\le  \frac{2}{\f(R_1)}\dbE \Big[f\big(|\Th^s_t - \tilde \Th^s_t|\big)\Big]. 
\eea
Denote $\d \Th^s_t : = \Th^s_t - \tilde\Th^s_t $. By the definition of the coupling above, we have
\beaa
d \d \Th^s_t = \Big(b^t(s, \Th^s_t) - \tilde b^t(s, \tilde \Th^s_t)\Big) ds + 2 \rc(\Si^s_t) \sigma d\bar W_s,
\eeaa
where $\bar W_s:=\int_0^s  e_r \cd dW^1_r $ is a one-dimensional Brownian motion. Denote $r_s := |\d \Th^s_t|$ and note that by the definition of $\rc$ we have $\rc (\Si^s_t) = 0$ whenever $r_s\le \e/2$. Therefore, one may show that
\beaa
d r_s = e_s \cd \Big( b^t(s, \Th^s_t)- \tilde b^t(s, \tilde \Th^s_t)\Big) ds + 2 \rc(\Si^s_t) \sigma d \bar W_s. 
\eeaa
Define $\cL^x_s$ as the right-continuous local time of $(r_s)$ and $\mu_f$ as the nonpositive measure representing the second derivative of $f$. Then it follows from the It\^o-Tanaka formula that
\beaa
 f(r_s) - f(r_0)
  &=& \int_0^s f'(r_u) e_u \cd \big( b^t(u, \Th^u_t)- \tilde b^t(u, \tilde \Th^u_t)\big) du 
 + \frac12 \int_{\dbR} \cL^x_s \mu_f (dx)+ M_s, \\
&\le & \int_0^s \left( f'(r_u) e_u \cd \big( b^t(u, \Th^u_t)- \tilde b^t(u, \tilde \Th^u_t)\big) + 2\rc(\Si^u_t)^2 \sigma^2 f''(r_u) \right) du 
+ M_s,
\eeaa
where $ M_s:= 2 \int_0^s \rc(\Si^u_t) f' (r_u) \sigma d\bar W_u $ is a martingale, and
the last inequality is due to the concavity of $f$. Now it is important to note that under the assumptions of the present Theorem we have
\beaa
 e_s \cd \big( b^t(s, \Th^s_t)- \tilde b^t(s, \tilde\Th^s_t)\big) 
 \le 1_{\{ r_s\ge \e\}} r_s \k(r_s) + 1_{\{r_s<\e\}} \gamma \e + \gamma \ol\cW^T_1(\nu^s, \tilde\nu^s),
\eeaa
where $ \gamma := K^2(1+K)\exp(2KT) $ with $ K $ a common Lipschitz coefficient of $ \nabla_aH $, $ \nabla_xH $, $ \nabla_xG $ and $ \phi $, which is computed explicitly in the proof of Theorem \ref{thm:well}. Further, since $f'' \le 0$ and $\rc(\Si^s_t) = 1$ whenever $r_s \ge\e$, we have
\beaa
f(r_s) - f(r_0) &\le& \int_0^s \Big( 1_{\{ r_u\ge \e\}}\big( f'(r_u) r_u \k(r_u)  + 2 \sigma^2 f''(r_u)\big)  +  1_{\{ r_u < \e\}} \gamma\e +  \gamma \ol\cW^T_1(\nu^u, \tilde\nu^u) \Big) du  + M_s\\
&\le & \int_0^s \Big( -1_{\{ r_u\ge \e\}}c\sigma^2 f(r_u)   +  1_{\{ r_u < \e\}} \gamma\e +  \gamma \ol\cW^T_1(\nu^u, \tilde\nu^u) \Big) du  + M_s \\
& = & \int_0^s \Big( -c\sigma^2 f(r_u)   +  1_{\{ r_u < \e\}} (\gamma+c \sigma^2)\e +  \gamma \ol\cW^T_1(\nu^u, \tilde\nu^u) \Big) du  + M_s 
\eeaa
The last inequality is due to \eqref{eq:fODE}. It clearly leads to
\beaa
\dbE \Big[ e^{c\sigma^2 s} f(r_s) - f(r_0)\Big] \le \int_0^s e^{c\sigma^2 u} \Big((\gamma+c \sigma^2) \e + \gamma \ol\cW^T_1(\nu^u, \tilde\nu^u) \Big) du.
\eeaa
Recall \eqref{eq:fWlower} and \eqref{eq:fWupper}. Together with the estimate above, we obtain
\bea\label{eq:estimate_coupling}
\frac{\f(R_1)}{2}e^{c\sigma^2 s} \cW_1 (\nu^s_t, \tilde\nu^s_t) - \cW_1 (\nu^0_t, \tilde\nu^0_t)  
\le \int_0^s e^{c\sigma^2 u} \Big(( \gamma + c \sigma^2 ) \e + \gamma \ol\cW^T_1(\nu^u, \tilde\nu^u) \Big) du.
\eea

\no{\it Step 2}.\q Since for each $t\in [0,T]$ one may obtain the estimate \eqref{eq:estimate_coupling} through the previous coupling argument, we have
\beaa
\frac{\f(R_1)}{2}e^{c\sigma^2 s} \ol\cW^T_1 (\nu^s, \tilde\nu^s) -  \ol\cW^T_1 (\nu^0, \tilde\nu^0) 
\le   T \int_0^s e^{c\sigma^2 u} \Big( ( \gamma + c \sigma^2 ) \e + \gamma \ol\cW^T_1(\nu^u, \tilde\nu^u) \Big) du.
\eeaa
By the Gronwall inequality, we have
\beaa
 \ol\cW^T_1 (\nu^s, \tilde\nu^s)
\le e^{\left(\frac{2\gamma T}{\f(R_1)}-c\sigma^2\right)s}  \frac{2}{\f(R_1)}  \ol\cW^T_1 (\nu^0, \tilde\nu^0)  +  \frac{2 T}{\f(R_1)} e^{\frac{2\gamma T}{\f(R_1)}s} ( \gamma + c \sigma^2 ) \e \int_{0}^{s} e^{ c \sigma^2 u } du.
\eeaa
This holds true for all $\e>0$, so finally we obtain \eqref{eq:contractionresult}.
\qed

\bibliographystyle{siamplain}
\bibliography{references}
	
\end{document}